\documentclass[a4paper,11pt]{article} 
\usepackage{amsmath,amsthm,amsfonts,amssymb,color} 
\usepackage[matrix,arrow]{xy} 
\usepackage{graphicx} 
\usepackage[english]{babel} 

\newtheorem{definition}{Definition}[section]
\newtheorem{corollary}[definition]{Corollary}
\newtheorem{example}[definition]{Example}
\newtheorem{lemma}[definition]{Lemma}
\newtheorem{proposition}[definition]{Proposition}
\newtheorem{theorem}[definition]{Theorem}
\newtheorem{remark}[definition]{Remark}
\newtheorem{remarks}[definition]{Remarks}
\newtheorem{notation}[definition]{Notation}

\newcommand{\hm}{\hspace*{.5cm}}

\newcommand{\huc}{\hspace*{.1cm}}

\newcommand{\n}{\neg} 

\newcommand{\su}{\mbox{$\huc\subseteq\huc$}}

\newcommand{\vaz}{{\emptyset}}                    

\newcommand{\bsq}{\blacksquare}

\newcommand{\bc}{\begin{center}}
\newcommand{\ec}{\end{center}}
\newcommand{\ol}{\overline}

\newcommand{\IF}{{I\kern-0.3emF}}
\newcommand{\IK}{{I\kern-0.3emK}}
\newcommand{\IN}{\mathbb{N}}
\newcommand{\IQ}{{I\kern-0.3emQ}}
\newcommand{\IR}{{I\kern-0.3emR}}

\newcommand{\K}{{\cal K}}
\newcommand{\lc}{{{\mathcal L}}}

\newcommand{\V}{\mathcal V}

\newcommand{\lra}{\mbox{$\longrightarrow$}}

\newcommand{\rar}{\rightarrow}

\newcommand{\lan}{\langle}
\newcommand{\ran}{\rangle}

\newlength{\chico}

\newcommand{\dps}{\displaystyle}

\begin{document}
\title{Fibring by functions as a method for combining matrix logics}
\author{V\'\i ctor L. Fern\'andez $^{\ast}$ and Marcelo E. Coniglio $^{\dag}$\\\\
\small $^{\ast}$ Basic Sciences Institute (Mathematical Area)\\
\small National University of San Juan\\
\small Av. Ignacio de la Roza 230 (O), San Juan, Argentina \\
\small E-mail:~{\tt vlfernan@ffha.unsj.edu.ar}\\
\small {$^{\dag}$} Centre for Logic, Epistemology and the History of Science (CLE) and\\ 
{\small Institute of Philosophy and the Humanities  (IFCH)}\\
\small University of Campinas (UNICAMP)\\ 
\small R. Cora Coralina 100, 13083-896, Campinas, SP, Brazil\\
\small E-mail:~{\tt coniglio@unicamp.br}\\
}
\maketitle

\begin{abstract}
We present in this paper an adaptation of the process of combination of logics known as fibring introduced by D. Gabbay. We are focused on the combination of two logics defined by matrix semantics, and based on pairs of functions that relate the logics to be combined. A number of technical results are proved. Among them, we demonstrate that the fibring of two matrix logics is also a matrix one. In addition, we prove that fibring is a (weak) conservative extension of the original logics, and we give conditions for such extension to be strong. We also study the case of fibring identifying two connectives as being the same. Several examples referred to fibring of some well-known matrix logics are shown along this paper.
\end{abstract}

\

\small \noindent {\em Keywords:} {Many-Valued Logics; Combination of Logics; Fibring.}

\small \noindent {\em MSC 2010:} {03B62; 03B50}

\normalsize

\section*{Introduction}

The method of combination of logics known as  {\em fibring} was introduced by D. Gabbay in the 90{'}s, with the aim of
combining logics having Kripke semantics, such as intuitionistic
and modal logics (see \cite{gab:96} and \cite{gab:99}).
The underlying idea of this technique is, roughly speaking, the following:
given two logics $\lc_1$ and $\lc_2$ admitting Kripke semantics
($Kr_1$ and $Kr_2$ respectively), a new logic $\lc_1 \circledast
\lc_2$ is defined over the language obtained from the union of
the connectives of both $\lc_1$ and $\lc_2$. So, the formulas of
$\lc_i$ ($i=1,2$) are also formulas of $\lc_1 \circledast \lc_2$.
But in the new logic there exist 
``hybrid formulas'' as we called them, obtained by mixing the connectives of the
two original logics. Thus, it is possible to evaluate a connective
$c$ of the language of $\lc_1$ applied to formulas of the
language of $\lc_2$. This can be done by means of a pair of functions: the first one
associates, to each world of any model of $Kr_1$, a world
of a model of $Kr_2$; meanwhile the second one works from $Kr_2$ to $Kr_1$. Taking into account  {\em all the possible pairs, of functions between $Kr_1$ and $Kr_2$}, the hybrid formulas can now be evaluated. Moreover, $\lc_1
\circledast \lc_2$ is a weak extension of $\lc_1$ and $\lc_2$, which means that the tautologies of the original logics are also tautologies of the new one.

Right after the introduction of fibring in the literature, the
original notion was investigated and modified by several authors.
In particular, an interesting adaptation of fibring using the
framework of Category Theory was introduced  in~\cite{ser:ser:cal:99}. Following~\cite{car:con:gab:gou:ser:08}, we will refer to it as {\em categorial} or {\em algebraic fibring}. This new approach to fibring was widely studied afterwards (see~\cite{car:con:07} for a survey on fibring and other methods for combining logics). In the present article, we return to the original definition of fibring (which, following~\cite{car:con:gab:gou:ser:08} once again, we will call it  {\em fibring by functions}), but in a modified presentation: instead of dealing with Kripke semantics, we will deal exclusively with logics  presented by means of  matrix semantics. The abstract theory of logical matrices was mainly developed by the Polish school of logic (see, for instance, \cite{bro:sus:73}, \cite{woj:84} and, more recently, \cite{fon:16}). In the case of finite matrices, it can be seen as an abstract version of the method of truth-tables for propositional logics.  In this paper we will work mainly with fibring applied to the case of finite matrices, indeed. 
It should be mentioned that the combination by fibring of logical matrices was already addressed in the literature, under different perspectives. A first approach from the point of view of fibring by functions was proposed by us in~\cite{con:fer:05} and~\cite{fer:con:14}. From the perspective of algebraic fibring, in~\cite{mar:car:17a} it was proven that a logic which is
uncharacterizable by a single  (finite or infinite) matrix could be obtained by means of the algebraic fibring of two  logics, each of one given by a single finite matrix. A possible solution to this problem was proposed in~\cite{mar:car:17b}, where the semantics of unconstrained (algebraic) fibring of logics given by nondeterministic matrix semantics was 
characterized.\footnote{Recall that a {\em nondeterministic matrix} (or simply an {\em Nmatrix}) is a logical matrix in which each connective is interpreted by means of a {\em multifunction}, that is, a function which, for any input, can produce a (nonempty) set of truth-values instead of a single truth-value.}

The method proposed here (which is based on~\cite{con:fer:05} and~\cite{fer:con:14}) is defined as follows. First, consider
two (propositional) logics $\lc_i$, defined by the signatures (sets of connectives) $C_i$ ($i = 1,2$), whose respective consequence relations $\models_{M_i}$ are defined by matrices $M_i = (A_i, D_i)$. To define a new ``mixed logic'',  its 
underlying language is
obtained from the original languages of both logics, with the addition of the hybrid formulas. Now, consider a fixed pair $(\lambda,\mu)$ of maps, $\lambda: A_1 \to A_2$, $\mu: A_2 \to A_1$. These functions will allow us to alternate between valuation (or truth-values) over $M_1$ and valuations (or truth-values) over $M_2$. Then, we can obtain
the truth-value of any hybrid formula of the mixed language. In this way, we can define a consequence relation on the mixed language which allows us to define a new logic, denoted by $[\lc_1 \circledast \lc_2]_{(\lambda,\mu)}$, called the  {\em fibring by functions $(\lambda,\mu)$} of $\lc_1$ and $\lc_2$.  In addition, we will study some interesting technical questions, such as the problem of the conservative extension (of the original consequence relations), and the identification of two connectives in the combined logic. 

\section{Basic Notions}

In this section we recall well-known definitions and results to be used throughout the paper. For references about this consult, for instance, \cite{woj:84}, \cite{bur:san:81} and~\cite{fon:16}.

\begin{definition}\label{definicion-basica}
Consider a fixed, countable set $\V = \{p_1, p_2, \ldots \}$ of symbols called {\em propositional variables}.\\
$(a)$ A {\em propositional
signature} is a family $C = \{C^k\}_{k \in \IN}$ of pairwise disjoint sets such that
$C^k\cap \V=\emptyset$ for every $k$. The  {\em
domain of the signature $C$} is the set $|C|:=\bigcup_{k \in \IN}
C^k$. Given two signatures $C_1$ and $C_2$, we say that {\em
$C_1$ is a subsignature of $C_2$} (indicated by $C_1 \sqsubseteq C_2$) if, for
every $k \in \IN$, $C^k_1 \su C^k_2$. The signatures $C_1 \cup
C_2$ and $C_1\uplus C_2$ (the union and the
disjoint union of $C_1$ and $C_2$, respectively) are defined as expected.\\
$(b)$ The  {\em propositional language over signature $C$} (denoted
by $L(C)$) is the absolutely free algebra of words generated by $C$ over
$\V$, considering each set $C^k$ as the set of $k$-ary operations
of that algebra. If $C_1 \sqsubseteq  C_2$ we say that $L(C_1)$ is  {\em a fragment of $L(C_2)$}. Elements in $L(C)$ are called {\em formulas}, while elements in $\V \cup C_0$ are called {\em atomic formulas}.
\end{definition}

\noindent For every formula $\varphi \in L(C)$, $\varphi(p_1,\dots,p_n)$ denotes that the propositional variables appearing in $\varphi$ belong to the set $\{p_1,\dots,p_n\}$.

An algebra over $C$ (or a $C$-algebra) will be denoted by ${\bf A} = (A,C^{\bf A})$, where $c^{\bf A}: A^k \to A$ is  the operation associated to $c \in C^k$. The set $C^{\bf A}$ will be often denoted by $C$, and every operation $c^{\bf A}$ will be denoted simply by $c$, if the context is clear enough.

\begin{definition} \label{c-matrices} {Given a signature $C$, a  {\em $C$-matrix} is a pair $M = ({\bf A},D)$, where ${\bf A} = (A,C^{\bf A})$ is a $C$-algebra, and $D \su A$. Every operation $c^{\bf A}$ of $C^{\bf A}$ will be called a  {\em truth-function of $M$}.} 
\end{definition}

\noindent Given a $C$-matrix $M = ({\bf A},D)$ and a formula $\varphi(p_1,\dots,p_n)$  we may write $\varphi^{\bf A}(x_1,\dots,x_n)$ or $\varphi^{\bf A}(\vec{x})$ to denote the truth function associated to $\varphi$ in {\bf A}. When there is no risk of confusion, the superscript in $\varphi^{\bf A}$ may be omitted.

\begin{definition}\label{valuaciones-matrices} Let $M = ({\bf A},D)$ be a $C$-matrix.\\
(1) An  {\em $M$-valuation} is a homomorphism $v:L(C) \lra A$.\\
(2) The {\em consequence relation 
$\models_{M} \su \wp(L(C)) \times L(C)$, induced by M} is given by: $(\Gamma,\varphi) \in \, \, \models_{M}$ iff, for every $M$-valuation $v$, 
if $v(\Gamma) \su D$, then $v(\varphi) \in D$. As usual, we denote $(\Gamma,\varphi) \in \, \, \models_{M}$ by $\Gamma \models_{M} \varphi$.\\
(3) Let $\K$ be a class of $C$-matrices. The  {\em consequence relation $\models_{\K}$} is defined as: $\Gamma \models_{\K} \varphi$ iff $\Gamma \models_{M} \varphi$ for every matrix $M$ of $\K$. That is, $\models_{\K} = \bigcap \limits_{M \in \K} \models_M$.
\end{definition} 

\noindent  Since $L(C)$ is an absolutely free algebra, an $M$-valuation is completely characterized by a function  $v:\V \lra A$.
 
Recall that a {\em Tarskian logic} is a pair $\mathcal{L}= \lan F,\vdash\ran$ such that  $F$ is a nonempty set (of {\em formulas}) and ${\vdash}\subseteq \wp(F)\times F$ is a {\em consequence relation} satisfying the following:~(i) $\varphi \in \Gamma$ implies that $\Gamma \vdash \varphi$; (ii)~if $\Gamma \vdash \varphi$ and $\Gamma \subseteq \Delta$ then $\Delta \vdash \varphi$; and~(iii)~if $\Gamma \vdash \varphi$, for every $\varphi \in \Delta$, and $\Delta \vdash \psi$ then $\Gamma \vdash \psi$. A logic $\mathcal{L}$ is said to be {\em finitary} if $\Gamma \vdash \varphi$ implies that  $\Gamma_0 \vdash \varphi$ for some finite $\Gamma_0 \subseteq \Gamma$; and  $\mathcal{L}=\lan L(C),\vdash\ran$ is {\em structural} if $\Gamma \vdash \varphi$ implies $\sigma(\Gamma) \vdash \sigma(\varphi)$ for every substitution $\sigma$ over $C$ (that is, every endomorphism $\sigma: L(C) \lra L(C)$;  again, since $L(C)$ is an absolutely free algebra, it is enough considering a function $\sigma: \V \lra L(C)$). Finally, $\mathcal{L}$ is {\em standard} if it is Tarskian, finitary and structural. 
An structural logic $\mathcal{L}=\lan L(C),\vdash\ran$ will be alternatively denoted by $\mathcal{L}=\lan C,\vdash\ran$.

Given $\varphi(p_1,\dots,p_n)$ and a substitution $\sigma$ such that $\sigma(p_i)=\alpha_i$,  $\sigma(\varphi)$ will be denoted by $\varphi(\alpha_1,\dots,\alpha_n)$ or even by $\varphi(\vec{\alpha})$, where  $\vec{\alpha}$ denotes $(\alpha_1,\dots,\alpha_n)$. 

\begin{remark} \label{matriz-es-estructural} It is immediate to prove that $\lc= \lan L(C),\models_{\K} \ran$ is Tarskian and structural. More than this, W\'ojcicki proves in {\rm~\cite{woj:70}} that every Tarskian and  structural logic over a signature $C$ can be characterized  by a class of logical matrices over $C$. In addition,  W\'ojcicki has shown in {\rm~\cite{woj:73}} that, if $\K$ is a finite class of finite matrices (that is, matrices with finite universes), then $\lc= \lan C,\models_{\K} \ran$ is finitary, and therefore standard. From now on, every  logic of the form $\lc = \lan C,\models_{\K} \ran$ will be called a {\em matrix logic}.
\end{remark}

\section{Unconstrained Fibring over Matrix Logics}

Consider two matrix logics $\lc_i = \lan C_i, \models_{M_i} \ran$ ($i = 1,2$). Now suppose that we want to define a new logic $\lc$, from $\lc_1$ and $\lc_2$, with the following properties:
\begin{itemize}
  \item[-] The signature $C$ of $\lc$ is $C_1 \uplus C_2$, the disjoint union of $C_1$ and $C_2$. 
  \item[-] $\lc$ is defined by a matrix logic $M$, being $M$ determined (in some reasonable way) by $M_1$ and $M_2$. 
  \end{itemize}
The fundamental problem here (that we could call ``the problem of the combinations of the languages'') arises from the conjunction of the requirements above indicated: if we have a formula $\varphi \in L (C_1 \uplus C_2 )$, and we wish to eva\-luate it, a first approach suggests to consider $A:= A_1 \uplus A_2$ as the support of $M$. The $M$-valuations would be, now, $M_1$-valuations or  $M_2$-valuations, according the internal structure of 
$\varphi$. Obviously, a new problem appears in this case: what to do when $\varphi$ is a hybrid formula, i.e., generated by connectives of both signatures $C_1$ and $C_2$? Clearly, the set of $M$-valuations should be defined considering at the same time the sets of $M_1$-valuations   and $M_2$-valuations, and with some kind of  interaction rules between them. A possible answer (of course, not the only one) we propose here is inspired by the so-called ``Gabbay approach'' to the problem of combinations of modal logics, and we called here the  {\em (unconstrained) fibring by functions}. Roughly speaking, this process is determined by  pairs $(\lambda,\mu)$ of functions; 
$\lambda: A_1 \lra A_2$, $\mu: A_2 \lra A_1$. These functions ``translate'', when necessary, the truth-values of $M_1$ and  $M_2$. In formal terms:

\begin{definition} \label{gfibringmatriz} Let $\lc_i = \langle
C_{i}, \models_{M_i} \rangle$ be two matrix logics determined by the $C_i$-matrices $M_i = ({\bf A}_i, D_i)$, where ${\bf A}_i = (A_i, C_i^{{\bf A}_i})$ ($i=1,2$) . The {\em fibred signature} (induced by $C_1$ and $C_2$) is $C_{\circledast}:= C_1 \uplus
C_2$. The language $L(C_{\circledast})$ will be called the {\em fibred language}. Every pair of functions $(\lambda, \mu) \in {A_2}^{A_1} \times {A_1}^{A_2}$ is called a  {\em fibring pair}.
\end{definition} 

\begin{definition}\label{sf-valuation}
(1) Let $(\lambda^1_2,\lambda^2_1)$ be  a fibring pair. A  {\em simple fibred valuation} (in short, {\em s.f.v}) is any map $v: \V \lra A_{\circledast}$, where
$A_{\circledast}:= A_1 \uplus A_2$. This function $v$ is extended to a function $\ol{v}: L(C_{\circledast}) \lra A_{\circledast}$ (which is determined by $v$ and $(\lambda^1_2,\lambda^2_1)$), according to the following recursive definition:
\begin{itemize}
\item If $\varphi \in \V$ then $\ol{v}(\varphi)$ =
  $v(\varphi)$;
  \item If $\varphi = c(\psi_1,\ldots,\psi_k)$ such that $c \in C_i^k$ then\footnote{Observe that there is a unique $i \in \{1,2\}$  such that $c \in C_i^k$.}
$$\ol{v}(\varphi) = c^{{\bf A}_i}(\bar{\lambda}^i_j(\ol{v}(\psi_1)),\ldots,\bar{\lambda}^i_j(\ol{v}(\psi_k))$$
where,  $\bar{\lambda}^i_j:A_{\circledast} \lra A_1 \cup A_2$ is defined as follows: $\bar{\lambda}^i_j(a)=a$ if $a \in A_i$, and $\bar{\lambda}^i_j(a)= \lambda^j_i(a)$ if $a \in A_j$.
\end{itemize}

\noindent (2) Let $D_{\circledast}$:=$D_1 \circledast D_2$. Consider  a fibring pair $(\lambda,\mu)$. An s.f.v  $\ol{v}$  {\em satisfies $\varphi$} if $\ol{v}(\varphi) \in D_{\circledast}$.\\
(3) The  {\em fibred consequence relation
$\models^{\circledast}_{(\lambda,\mu)} \su \wp(L(C_{\circledast})) \times
L(C_{\circledast})$} is defined as follows: $\Gamma \models^{\circledast}_{(\lambda,\mu)} \varphi$ if, for every s.f.v $\ol{v}$ satisfying simultaneously every formula in $\Gamma$, it is the case that $\ol{v}$ satisfies $\varphi$. Finally, the pair
$[\lc_1 \circledast \lc_2]_{(\lambda,\mu)} = \lan C_{\circledast}, \models^{\circledast}_{(\lambda,\mu)}\ran$ will be called the (weak and unconstrained)\footnote{The meaning of the expressions ``weak'' and ``unconstrained'' will be clear later. The notation $[\lc_1 \circledast \lc_2]_{(\lambda,\mu)} = \lan C_{\circledast}, \models^{\circledast}_{(\lambda,\mu)}\ran$ will be clear after Corollary~\ref{fibring-estru-finitario}.}  {\em fibring by functions  $(\lambda,\mu)$ of
$\lc_1$ and $\lc_2$}.
\end{definition}

\begin{example} \label{P1-1}
Consider the paraconsistent logic $P^1$, introduced
in {\rm~\cite{set:73}}. Its signature is such that $|C_{P^1}|$
= $\{\n_{P^1}, \rar_{P^1} \}$, with  arities as expected. The logic $P^1$ is determined by the matrix $M_{P^1}$ =
$({\bf A}_{P^1}, \{T,T_1\})$ such that ${\bf A}_{P_1}$ has the support $A_{P^1}=\{T, T_1,
F\}$, and the truth-functions corresponding to $\n_{P^1}$ and
$\rar_{P^1}$ are defined through the tables below.

\

$$\begin{array}{|c|c| c| c|} \hline & T & T_1 & F \\
\hline \neg_{P^1} & F & T & T  \\ \hline \end{array} \hspace{2cm}
\begin{array}{|c|c |c |c|} \hline \rar_{P^1} & T & T_1 & F \\ \hline T &
T & T & F\\ \hline T_1 & T & T & F\\ \hline F & T & T & T \\
\hline
\end{array} \hm$$

\

\noindent Now, let $CPL$ be the classical propositional logic,
defined over the signature $C_{CPL}$ such that $|C_{CPL}|$ =
$\{\wedge_{CPL}, \neg_{CPL}\}$ (with usual arities). The matrix semantics for $CPL$ is based on the well-known matrix ${M}_{CPL}:= ({\bf 2},\{1\})$, where ${\bf 2}$ = 
$(2,C_{CPL})$, $2:= \{0,1\}$, and the truth-functions related to $\wedge_{CPL}$ and to $\n_{CPL}$ are the usual. Now, let $(\lambda,\mu)$ be a fibring a
pair defined by:

\begin{itemize}
  \item $\lambda(T)=1$; $\lambda(T_1) = 1$; $\lambda(F) = 0$; 
\item $\mu(1)=T$; $\mu(0)=F$.
\end{itemize}

\noindent Consider the (hybrid) formula $\varphi$ :=
$p \rar_{{\!\!\!\!}_{P^1}} (\neg_{{\!\!}_{P^1}} r \wedge_{{\!}_{CPL}}(q \wedge_{{\!}_{CPL}} \neg_{{\!}_{CPL}} r))$ (using standard infix notation) and let  $v$ be an s.f.v. such that $v(p) = T$, $v(q) = F$ and $v(r) = 0$. From this we have, by Definition~\ref{sf-valuation}:

\

\noindent $\begin{array}{ll} \overline{v}(\varphi) & = 
T {\rar_{{\!\!\!\!}_{P^1}}} \mu (\lambda(\neg_{{\!\!}_{P^1}} \mu(0)) \wedge_{{\!}_{CPL}}(\lambda(F) \wedge_{{\!}_{CPL}} \neg_{{\!}_{CPL}} 0))\\
& =  T \rar_{{\!\!\!\!}_{P^1}} \mu (\lambda (\neg_{{\!\!}_{P^1}} F) \wedge_{{\!}_{CPL}} (0 \wedge_{{\!}_{CPL}} 1))\\
& =  T \rar_{{\!\!\!\!}_{P^1}} \mu (\lambda (T) \wedge_{{\!}_{CPL}} 0)\\
& = 
T \rar_{{\!\!\!\!}_{P^1}} \mu (1 \wedge_{{\!}_{CPL}} 0)\\
& = 
T \rar_{{\!\!\!\!}_{P^1}} \mu (0) = T \rar_{{\!\!\!\!}_{P^1}} F = F.
\end{array}
$ 

\hfill{$\bsq$}
\end{example}

\begin{remark}\label{no-es-homom} The functions $\ol{v}: L(C_{\circledast}) \lra A_{\circledast}$ introduced in Definition \ref{sf-valuation}  {\em cannot be considered as homomorphisms} between the involved algebras. This is because the operations associated to the connectives of $C_{\circledast}$ are just partially defined in $A_{\circledast}$, as the following example shows.
\end{remark}

\begin{example} \label{ejemplo-no-homomorfismo} Consider the same logics of Example \ref{P1-1}, and the same fibring pair $(\lambda,\mu)$. Let $v$ be an s.f.v. such that $v(p) = 1$ and 
$v(q) = F$. Since $\rar_{P^1} \in C_1$ and $\ol{v}(p)=v(p)=1\in A_2$, then $\bar{\lambda}(\ol{v}(p)) = \mu(\ol{v}(p))=\mu(1) = T$. On the other hand, since $\ol{v}(q)=v(q)=F \in A_1$ then $\bar{\lambda}(\ol{v}(q)) = \ol{v}(q) = F$.
Then, $\ol{v}(p \rar_{P^1} q) =  \bar{\lambda}(\ol{v}(p)) \rar_{P^1} \bar{\lambda}(\ol{v}(q)) = T \rar_{P^1} F = F \in A_{\circledast}$. However, $\ol{v}(p) \rar_{P_1} \ol{v}(q)$ is not defined, since the domain of $\rar_{P^1}$ is $A_1$, and $\ol{v}(p)  = 1 \in A_2$. \hfill{$\bsq$}
\end{example}

\noindent Until now, from Remark \ref{no-es-homom} and the example above, the fibring $[\lc_1 \circledast \lc_2]_{(\lambda,\mu)}$, as given in Definition \ref{gfibringmatriz}, cannot be considered as a matrix logic. This is because, in part, we have followed the ideas and formalism of \cite{con:fer:05} and \cite{gab:96}. However, we will see that, in a somewhat hidden way, the relation $\models^{\circledast}_{(\lambda,\mu)}$ can be characterized by a matrix. This result will be developed in the sequel.

\begin{definition} \label{funciones-lambda-mu} Let $C_1$, $C_2$ be two signatures, and let $M_i = ({\bf A}_i,D_i)$ be a $C_i$-matrix ($i = 1,2$). 
For every fibring pair $(\lambda,\mu)$, we define the  {\em functions} $\ast_{\lambda}: A_{\circledast} \lra A_2$ and $\ast_{\mu}: A_{\circledast} \lra A_1$ as follows: \\

$\begin{array}{lll}
\ast_{\lambda}(x)=\left \{\begin{array}{ll}  \lambda(x)  &  \mbox{if $x \in A_1$} \\ 
\\
x &  \mbox{if $x \in A_2$}\end{array} \right. & \ \ \ \ &
\ast_{\mu}(x)=\left \{\begin{array}{ll} x  &  \mbox{if $x \in A_1$} \\ 
\\
\mu(x) & \mbox{if $x \in A_2$}\end{array} \right.
\end{array}$
\end{definition}

\

\begin{definition} \label{matriz-lambda-mu} For $i = 1,2$, consider $C_i$, $M_i = ({\bf A_i},D_i)$, $C_{\circledast}$ and $(\lambda,\mu)$,  as in Definition \ref{gfibringmatriz}. The  {\em fibred $C_{\circledast}$-matrix} $(M_1 \circledast M_2)_{(\lambda,\mu)} := ({\bf A_{\circledast}^{(\lambda,\mu)}},D_{\circledast})$ is such that ${\bf A_{\circledast}^{(\lambda,\mu)}}$ := $(A_{\circledast}, C_{\circledast}^{(\lambda,\mu)})$, where: the sets $A_{\circledast}$ and $D_{\circledast}$ are the same of Definition~\ref{sf-valuation},  and the truth-functions of $C_{\circledast}^{(\lambda,\mu)}$ are defined as follows (in the sequel, we consider $k$-ary connectives, and $k$-tuples $\vec{x} = (x_1,\dots,x_k) \in (A_{\circledast})^k$):
\begin{itemize}
  \item [] If $c \in C_1^k$, then $c^{{\bf A}_{\circledast}^{(\lambda,\mu)}}(\vec{x}):= c^{{\bf A}_1}(\ast_{\mu}(x_1),\dots,\ast_{\mu}(x_k))$
\item [] If $c \in C_2^k$, then $c^{{\bf A}_{\circledast}^{(\lambda,\mu)}}(\vec{x}):= c^{{\bf A}_2}(\ast_{\lambda}(x_1),\dots,\ast_{\lambda}(x_k))$ 
\end{itemize}

\end{definition}

\noindent To simplify the previous notation, the truth-functions $c^{{\bf A}_{\circledast}^{(\lambda,\mu)}}$ will be often denoted merely by $c$.

\begin{example} \label{ejemplo-matrices-fibradas} Again, consider the signatures $C_1$, $C_2$, and the matrices $M_{P^1}$ and $M_{CPL}$ from Example \ref{P1-1}, with the same fibring pair $(\lambda,\mu)$. As we have seen, $|C_{\circledast}| = |C_1| \uplus |C_2| = \{ \rar_{{\!\!\!\!}_{P^1}}, \wedge_{{\!}_{CPL}}, 
\neg_{{\!\!}_{P^1}}, \neg_{{\!}_{CPL}}\}$ and $A_{\circledast}$ = 
$\{T,T_1,F,0,1\}$. The functions $\ast_{\mu}$ and $\ast_{\lambda}$ are defined as:

$$\begin{array}{|c|c| c| c|c|c|} \hline & T & T_1 & F & 1 & 0\\
\hline \ast_{\mu} & T & T_1 & F & T & F  \\ \hline \end{array} \hspace{2cm}
\begin{array}{|c|c| c| c|c|c|} \hline & T & T_1 & F & 1 & 0\\
\hline \ast_{\lambda} & 1 & 1 & 0 & 1 & 0  \\ \hline 
\end{array} \hm$$

\noindent These functions allow us to define the $C_{\circledast}$-matrix $(M_{P^1} \circledast M_{CPL})_{(\lambda,\mu)} = ({\bf A_{\circledast}^{(\lambda,\mu)}}, D_{\circledast})$, where $D_{\circledast} = \{T,T_1,1\}$. The truth-functions of  ${\bf A_{\circledast}^{(\lambda,\mu)}}$ are:

$$\begin{array}{|c|c| c| c|c|c|} \hline 
\rar_{{\!\!\!\!}_{P^1}} & T & T_1 & F & 1 & 0\\
\hline
T & T & T & F & T & F\\
T_1 & T & T & F & T & F\\
F & T & T & T & T & T\\
1 & T & T & F & T & F\\
0 & T & T & F & T & F\\
\hline \end{array}
\hspace{2cm}
\begin{array}{|c|c| c| c|c|c|} \hline 
\wedge_{{\!}_{CPL}} & T & T_1 & F & 1 & 0\\
\hline
T & 1 & 1 & 0 & 1 & 0\\
T_1 & 1 & 1 & 0 & 1 & 0\\
F & 0 & 0 & 0 & 0 & 0\\
1 & 1 & 1 & 0 & 1 & 0\\
0 & 0 & 0 & 0 & 0 & 0\\
\hline \end{array}
$$

\noindent (For $\neg_{{\!\!}_{P^1}}$ and $\neg_{{\!}_{CPL}}$ proceed in a similar way). \hfill $\bsq$ 
\end{example}

Note that every pair $(\lambda,\mu)$ determines a different $C_{\circledast}$-matrix $M_{(\lambda,\mu)}$, even when all these $C_{\circledast}$-matrices have the same support/set of designated values (that is, $A_{\circledast}$/$D_{\circledast}$). However, the truth-functions in $C_{\circledast}^{(\lambda,\mu)}$ depend of every pair $(\lambda,\mu)$. 

Besides that, consider a fixed matrix $M_{(\lambda,\mu)}$ and note that, in the definition of a truth-table associated to  a certain connective $c$, the output values of its associated truth-function always belong to the matrix whose signature contains $c$. In addition, note that the functions $\ast_{\mu}$ and $\ast_{\lambda}$ are $1$-ary truth-functions. However, usually there are no connectives in $C_{\circledast}$ associated to them. So we just can work with them in the context of the algebra ${\bf A_{\circledast}^{(\lambda,\mu)}}$, but we cannot consider them as belonging to the formal fibred language.

Definition \ref{matriz-lambda-mu} and Example \ref{ejemplo-matrices-fibradas} are strongly suggesting that the fibring of logics can be described, in fact, by the fibred matrices. The formal proof of this fact is given in the sequel. Actually, the relations $\models^{\circledast}_{(\lambda,\mu)}$ and $\models_{M_{(\lambda,\mu)}}$ are both defined interpreting the formulas in the same domain $A_{\circledast}$, and considering a special subset $D_{\circledast} \su A_{\circledast}$. The only difference in both consequence relations is the nature of the functions involved (s.f.v and $M_{(\lambda,\mu)}$-valuations, respectively, according Definitions \ref{c-matrices} and \ref{matriz-lambda-mu} above). Anyway, these functions can be interdefined, as the next technical result shows.

\begin{proposition}\label{fibring-matriz-tecnica} 
Let $\lc_i = \lan C_i, \models_{M_i}\ran$ be matrix logics, with $M_i = ({\bf A}_i, D_i)$; ${\bf A}_i = (A_i,C_i^{{\bf A}_i})$ $(i=1,2)$, and let $(\lambda,\mu)$ be a fibring pair. Then:

\noindent $(a)$ For every s.f.v. $v: \V \lra A_{\circledast}$ there is a $M_{(\lambda,\mu)}$-valuation $w_v: \V \lra A_{\circledast}$ such that, for every 
$\varphi \in L(C_{\circledast})$, $v(\varphi) = w_v(\varphi)$.

\noindent $(b)$ For every $M_{(\lambda,\mu)}$-valuation $w$ exists a simple fibred valuation $v_w$ such that, 
for every $\varphi \in L(C_{\circledast})$, $v(\varphi) = w_v(\varphi)$.
\end{proposition}
\begin{proof}
For $(a)$: let $v$ be a s.f.v. We define $w_v: \V \lra A_{\circledast}$ as $w_v(p) = v(p)$ for every $p \in\V$. This function is extended homomorphically to every formula of $L(C_{(\circledast})$ because $\V$ generates $L(C_{\circledast})$. Let us prove by induction on the complexity of $\varphi$ that $\ol{v}(\varphi)$ = 
$w_v(\varphi)$ for every $\varphi \in L(C_{\circledast})$. This is valid if $\varphi$ is a variable $p \in \V$, by definition. Now, suppose that $\varphi$ = 
$c(\psi_1,\dots,\psi_k)$, with $c$ a $k$-ary connective. Then:

\noindent $(a.1)$ $c \in C_1^k$. So, $w_v(\varphi) = c(w_v(\psi_1),\dots,w_v(\psi_k)) = c(\ol{v}(\psi_1),\dots,\ol{v}(\psi_k))$, by I. H. Now, 
taking into account Definition \ref{matriz-lambda-mu}, we have $c(\ol{v}(\psi_1),\dots,\ol{v}(\psi_k)) = c(\ast_{\mu}(\ol{v}(\psi_1)),\dots,\ast_\mu(\ol{v}(\psi_k)))$. Note now that, for $i$ =$1,\dots,k$, $\ast_{\mu}(\ol{v}(\psi_i)) = ({\psi_i})_{(v,c)}$, cf. Definition \ref{sf-valuation}. From 
this, $w_v(\varphi) = c(\ol{v}(\psi_1)_{(v,c)},\dots,
\ol{v}(\psi_k)_{(v,c)}) = \ol{v}(\varphi)$.

\noindent $(a.2)$ $c \in C_2^k$. Similar to the previous case, using the map $\ast_{\lambda}$. This completes the proof for $(a)$.

\noindent For $(b)$: if $w$ is a $M_{(\lambda,\mu)}$-valuation, define $v_w: \V \lra A_{(\lambda,\mu)}$ as $v_w(p) = v(p)$ for every $p \in\V$, and extend this function (using Definition \ref{gfibringmatriz}) to $\ol{v_w}: L(C_{\circledast}) \lra A_{\circledast}$. Then, proceed as in $(a)$.
\end{proof}

\noindent Note that the previous result has an equivalent formulation (whose proof is somewhat different to the one developed above), which is the following: for every map 
$v:L(C_{\circledast}) \to A_{\circledast}$ it holds that $v$ is a simple fibred valuation if and only if $v$ is an $M_{(\lambda,\mu)}$-valuation.

From Proposition \ref{fibring-matriz-tecnica} is is obvious the following:

\begin{theorem} \label{fibring-es-matricial} Let $\lc_i = \lan C_i, \models_{M_i}\ran$ be matrix logics, for $i=1,2$. Then, $[\lc_1 \circledast \lc_2]_{(\lambda,\mu)}$ coincides with the matrix logic $\lan C_{\circledast}, \models_{(M_1 \circledast M_2)_{(\lambda,\mu)}}\ran$, for every fibring pair $(\lambda,\mu)$.
\end{theorem}

\begin{corollary}\label{fibring-estru-finitario} For every pair of matrix logics $\lc_i = \lan C_i, \models_{M_i}\ran$ 
($i=1,2$):\\[1mm]
$(a)$ $[\lc_1 \circledast \lc_2]_{(\lambda,\mu)}$ is a structural logic.\\
$(b)$  If $M_1$ and $M_2$ are finite then
$\models_{(M_1 \circledast M_2)_{(\lambda,\mu)}}$ is finitary and so
$[\lc_1 \circledast \lc_2]_{(\lambda,\mu)}$ is a standard logic.
\end{corollary}

\noindent From now on, the fibring by functions of two logics will be exclusively represented by the matrix construction given in  Theorem \ref{fibring-es-matricial}, namely  $[\lc_1 \circledast \lc_2]_{(\lambda,\mu)} := \lan C_{\circledast}, \models_{(M_1 \circledast M_2)_{(\lambda,\mu)}}\ran$.

\section{Some basic properties of Fibring of logics}

Now we will analyze the relationship between the 
logics $\lc_i = \lan C_i, \models_{M_i}\ran$ ($i=1,2$), and the fibring $[\lc_1 \circledast \lc_2]_{(\lambda,\mu)}$. To lighten notation, we will write $M_{(\lambda,\mu)}$ instead of $(M_1 \circledast M_2)_{(\lambda,\mu)}$. First of all, we will give some general definitions:

\begin{definition}\label{extension-debil} Let  $\lc_1 = \lan C_1,\vdash_1\ran$ and $\lc_2 = \lan C_2, \vdash_2\ran$  be structural  logics.\\
(1) $\lc_2$ is a  {\em weak extension} of $\lc_1$ if
$C_1 \sqsubseteq C_2$ and, for every $ \varphi \in L(C_1)$, $ \vdash_1 \varphi$ implies that $\vdash_2 \varphi$.\\
(2) $\lc_2$ is  a  {\em conservative (weak) extension} of $\lc_1$ if, for every $\varphi \in L(C_1)$, $\vdash_1 \varphi$  {\em if and only if} $\vdash_2 \varphi$.
\end{definition}

\noindent The previous notion is usually generalized to a stronger one:

\begin{definition} \label{extension-fuerte}
 Let  $\lc_1 = \lan C_1,\vdash_1\ran$ and $\lc_2 = \lan C_2, \vdash_2\ran$  be structural  logics.\\
(1) $\lc_2$ is a  {\em strong extension} of $\lc_1$ if
$C_1 \sqsubseteq C_2$ and, for every $\Gamma \cup \{\varphi\}$, $\Gamma \vdash_1 \varphi$ implies that $\Gamma \vdash_2 \varphi$.\\
(2) $\lc_2$ is  a  {\em conservative strong extension} of $\lc_1$ if, for every $\Gamma \cup \{\varphi\}$, $\Gamma \vdash_1 \varphi$  {\em if and only if} $\Gamma \vdash_2 \varphi$.
\end{definition} 

\begin{lemma} \label{lema1} Let $w$ be an $M_{(\lambda,\mu)}$-valuation, and consider the $M_1$-valuation $w_1:\V \lra A_1$ such that $w_1(p)= \ast_{\mu}(w(p))$ for every $p \in \V$. Then, for every $\varphi\in L(C_1)\setminus\V$, $w(\varphi) = w_1(\varphi)$. An analogous result
holds for $M_2$, by defining the $M_2$-valuation $w_2$ as follows:  $w_2(p)= \ast_{\lambda}(w(p))$, for every $p \in \V$.
\end{lemma}
\begin{proof}
Use induction on the number of connectives appearing in $\varphi$, and apply
Definitions~\ref{funciones-lambda-mu} and \ref{matriz-lambda-mu}.
\end{proof}

\begin{lemma} \label{lema2} For every $M_1$-valuation $v_1:L(C_1)\lra A_1$, the $M_{(\lambda,\mu)}$-valuation $w_{v_1}:\V\lra A_{\circledast}$ defined by 
$w_{v_1}(p)=v_1(p)$ for every $p \in \V$ verifies that $w_{v_1}(\varphi) = v_1(\varphi)$ for every $\varphi\in
L(C_1)$. A similar result holds for $M_2$-valuations.
\end{lemma}
\begin{proof}
Again, it follows straightforwardly by induction on the
complexity of $\varphi$ and Definitions~\ref{funciones-lambda-mu} and \ref{matriz-lambda-mu}.
\end{proof}

\begin{proposition} \label{weak} Let $\lc_i = \lan C_i, \models_{M_i}\ran$ be non-trivial matrix logics\footnote{A logic $\lc = (C,\vdash)$ is  {\em trivial} if $\Gamma\vdash\varphi$ for
every $\Gamma\cup\{\varphi\} \su L(C)$.} 
($i=1,2$). Then $[\lc_1 \circledast \lc_2]_{(\lambda,\mu)}$ is a weak conservative extension
of $\lc_1$ and $\lc_2$. 
\end{proposition}
\begin{proof}
Since $\lc_i$ is a structural and non-trivial logic,
it follows that $\not\models_{M_i} p$ ($i=1,2$), for every $p \in \V$. 
From this, if
$\models_{M_1} \varphi$, then $\varphi \in L(C_1) \setminus \V$. Now, let $w$ be an 
$M_{(\lambda,\mu)}$-valuation. By Lemma~\ref{lema1} there exists an
$M_1$-valuation $w_1$ such that $w(\varphi)=w_1(\varphi)$. 
But $\models_{M_1} \varphi$ and so
$w_1(\varphi)\in D_1\su D_{\circledast}$. Thus, $w(\varphi)\in D_{\circledast}$ and then
$\models_{M_{(\lambda,\mu)}} \varphi$. Conversely, suppose $\varphi \in L(C_1)$, with
$\models_{M_{(\lambda,\mu)}} \varphi$, and let $v_1$ be an
$M_1$-valuation. The mapping $w_{v_1}:\V\lra A$ defined as in
Lemma~\ref{lema2} verifies $w_{v_1}(\varphi) = v_1(\varphi)$, and therefore $v_1(\varphi)\in D_{\circledast}$.
Moreover, $v_1(\varphi) \in D_1$, because $\varphi \in L(C_1)$. Hence, $\models_{M_1} \varphi$. From all this, $[\lc_1 \circledast \lc_2]_{(\lambda,\mu)}$ is a conservative weak extension of $\lc_1$. The proof for $\lc_2$ is
similar.
\end{proof}

\noindent Note that, if exactly one of the logics (say, $\lc_1$) is trivial,
then the last result is no longer true for any pair $(\lambda,\mu)$. In fact: assuming that
$\lc_1$ is trivial and $\lc_2$ is not trivial then, for every  $p \in \V$ it holds that $\models_{\lc_1} p$ and  $\not\models_{\lc_2}p$. Let $v_2$ be an $M_2$-valuation such that $v(p) \in A_2\setminus D_2$. Then, the $M_{(\lambda,\mu)}$-valuation $w_{v_2}$ defined as in Lemma~\ref{lema2} is such that $w_{v_2}(p) \notin D_{\circledast}$. From this, $\not\models_{M_{(\lambda,\mu)}} p$ and so $[\lc_1 \circledast \lc_2]_{(\lambda,\mu)}$ is not a weak extension of $\lc_1$.

Proposition~\ref{weak} shows that, for non-trivial logics, the new tautologies produced in $[\lc_1 \circledast \lc_2]_{(\lambda,\mu)}$ can only be hybrid formulas. However, this result cannot  be extended  to the case of strong extensions (recall Definition~\ref{extension-fuerte}) as the following example shows:

\begin{example} \label{fuertenovale}
Let $\lc_1 = \langle \{\vee_{CL} \}, \models_1 \rangle$ be the
disjunction-fragment of the classical propositional logic
(induced by the matrix ${M}_{CPL}:= ({\bf 2},\{1\})$ already seen on Example \ref{P1-1}
(using the truth-function for $\vee$ in this case). On the other hand, let
$\lc_2 = \langle C_2, \models_2 \rangle$ be any logic such that
$\models_2$ is defined by a matrix $M_2 = \lan {\bf A}, \{T,
T_1\}\ran$, with support $\{T, T_1, F\}$ (the signature $C_2$ and the operations of ${\bf A}$ are
irrelevant here). Obviously, $p_1 \models_1 p_1 \vee p_2$. Now,
consider a pair $(\lambda,\mu)$ where $\mu: \{T, T_1, F\} \lra \{0,
1\}$ is such that $\mu(T) = 0$ and $\lambda: \{0, 1\} \lra
\{T, T_1, F\}$ is arbitrary. Let $v: \V \lra \{0, 1, T, T_1, F \}$ be any $M_{(\lambda,\mu)}$-valuation such that
$v(p_1) = T$ and $v(p_2) = 0$. So, $v(p_1) \in D_{\circledast}$. On
the other hand, $v(p_1 \vee p_2) = v(p_1) \vee^{{\bf A}_{\circledast}^{(\lambda,\mu)}} v(p_2) = \ast_{\mu}(v(p_1)) \vee_{CL} \ast_{\mu}(v(p_2)) = \mu(T) \vee 0 = 0 \vee 0 = 0 \notin D_{\circledast}$. Hence, 
$p_1 \not\models_{M_{(\lambda,\mu)}} p_1 \vee p_2$. \hfill $\bsq$
\end{example}

\noindent The failure of conservativity of $M_{(\lambda,\mu)}$ with respect to $M_1$ is due to the fact that $\mu$ assigns a non-designated value to a designated value. In order to guarantee strong conservativity by means of fibring by functions, it is sufficient to consider {\em admissible pairs}.

\begin{definition} \label{gfibringrestri} {Let $\lc_i = \langle
C_{i}, M_i \rangle$ (with $i=1,2$) be two matrix logics. A fibring pair $(\lambda,\mu) \in A_2 ^{A_1}
\times A_1 ^{A_2}$ is  {\em admissible} if it satisfies: $\lambda(x) \in
D_2$ iff $x\in D_1$, for every $x\in A_1$; and $\mu(y)\in  D_1$ iff
$y \in D_2$, for every $y\in A_2$. If $(\lambda,\mu)$ is an admissible pair, then 
the fibring $[\lc_1 \circledast \lc_2]_{(\lambda,\mu)}$ will be denoted as $[\lc_1 \odot \lc_2]_{(\lambda,\mu)}$ and it will be called the  {\em strong fibring by functions $(\lambda,\mu)$ of $\lc_1$ and $\lc_2$}.}
\end{definition}

\noindent The condition of admissibility is related to the notion of  {\em matrix strict homomorphisms}
(cf.~\cite{cze:01}). This kind of homomorphisms will be used  in Section~\ref{fibring-sharing}.

\begin{definition}\label{compatibles} {Two logics $\lc_1$ and $\lc_2$ as in
Definition~\ref{gfibringrestri} are said to be  {\em compatible}
if there is at least one admissible pair $(\lambda,\mu)$ in $A_2 ^{A_1} \times A_1
^{A_2}$. Note that $\lc_1$ and $\lc_2$ are compatible iff:

$(i)$ $D_1\not=\vaz$ iff $D_2\not=\vaz$; and

$(ii)$ $(A_1\setminus D_1)\not=\vaz$ iff $(A_2\setminus
D_2)\not=\vaz$.}
\end{definition}
\noindent Observe that, if $\lc_1$ and $\lc_2$ are not compatible,
then one of the logics is trivial; therefore any pair of
nontrivial logics is compatible. Now, we will prove that
Proposition~\ref{weak} can be improved by considering strong
fibrings.

\begin{proposition} \label{extendcons}  Let $\lc_i = \langle
C_{i}, M_i \rangle$ (with $i=1,2$) be two matrix logics, and let $(\lambda,\mu)$ be an admissible pair. Then
$[\lc_1 \odot \lc_2]_{(\lambda,\mu)}$ is a strong, conservative extension
of both logics.
\end{proposition}
\begin{proof}
We will proof our claim only for $\lc_1$. Let $(\lambda,\mu)$ be admissible. If $w$ is an $M_{(\lambda,\mu)}$-valuation, consider $w_1$ the $M_1$-valuation defined
as in Lemma~\ref{lema1}. It can be proved that, for every $\varphi\in L(C_1)$:
$$w(\varphi)\in D_{\circledast} \ \textrm{ iff } \ w_1(\varphi)\in D_1.
\hspace{1cm} (*)$$ In fact: if $\varphi\in\V$ and
$w(\varphi)\in A_1$, then
$w_1(\varphi) = w(\varphi)$ and the result holds. On the other hand, if ${w}(\varphi) \in A_2$ then
$w_1(\varphi)=\ast_{\mu}(w(\varphi)) = \mu(w(\varphi))\in D_1$, if and only if
$w(\varphi) \in D_2 \su D_{\circledast}$. In addition,
if $\varphi\not\in\V$ then the result follows from
Lemma~\ref{lema1}. So, $(*)$ is valid.   Finally, by using  $(*)$ and Lemmas~\ref{lema1} and~\ref{lema2} it is immediate to prove that  $[\lc_1 \odot \lc_2]_{(\lambda,\mu)}$ is a conservative extension of $\lc_1$. The proof for  $\lc_2$ is analogous.
\end{proof}

\noindent Now,  two examples of fibring by admissible pairs will be given. The first one deals with  a generalization of the logic $P^1$.

\begin{example} \label{InPk} In {\rm ~\cite{fer:01}} and {\rm ~\cite{fer:18}} the hierarchy $\{P^k\}_{k\in\IN}$ of paraconsistent logics 
was developed.
The logics of this family are considered as a generalization of Sette{'}s logic $P^1$, and they are defined over $C_{P^k}$ =
$\{\neg_{{}_{P^k}}, \rar_{{}_{P^k}} \}$, with semantics given by the matrix
$M_{P^k} = ({\bf A_{P^k}}, \{T_0, T_1,\ldots,T_k\})$, where the support of ${\bf A_{P^k}}$
is $A_{P^k}=\{T_0, T_1,\ldots,T_k, {\bf f} \}$. The corresponding
operations are displayed in the tables below.

\

$$\begin{array}{|c|c| c| c|} \hline  & T_0 & T_h & {\bf f} \\
\hline \neg_{{}_{P^k}} & {\bf f} & T_{h-1} & T_0\\  \hline \end{array}
\hspace{1cm}
\begin{array}{|c|c| c| c|}  \hline \rar_{{}_{P^k}} & T_0 & T_{h} & {\bf f} \\ \hline T_0 &
T_0 & T_0 & {\bf f} \\ \hline T_h & T_0 & T_0 & {\bf f} \\ \hline
{\bf f} & T_0 & T_0 & T_0 \\ \hline
\end{array} \hspace{1cm} (1\leq h \leq k)$$

\

\noindent Also in {\rm ~\cite{fer:01}} and {\rm ~\cite{fer:18}}, the hierarchy $\{I^n\}_{n\in\IN}$ of
 {\em weakly-intuitionistic logics} was introduced,
generalizing the  logic $I^1$ introduced in
{\rm ~\cite{set:car:95}} as a dual of $P^1$. Each $I^n$-logic is defined over
the signature $C_{I^n} = \{\neg_{{}_{I^n}}, \rar_{{}_{I^n}} \}$, with
semantics given by the matrix $M_{I^n} = ({\bf A}_{I^n},
\{ {\bf t} \})$, being its support the set $A_{I^n}=\{{\bf t}, F_0,
F_1,\ldots,F_n\}$ \footnote{The notation of the sets $A_{P^k}$ and $A_{I^n}$ differs from the one used in \cite{fer:01} and \cite{fer:18}, to emphasize that such sets of truth-values are disjoint ones.}. The operations of the matrix $M_{I^n}$ are
given by the tables below.

\

$$\begin{array}{|c|c| c| c|} \hline   & {\bf t} & F_0 & F_l \\
\hline \neg_{{}_{I^n}} & F_0 & {\bf t} & F_{l-1}\\ \hline  \end{array}
\hspace{1cm}
\begin{array}{|c|c| c| c|} \hline \rar_{{}_{I^n}} & {\bf t} & F_0 & F_l \\ \hline
{\bf t} & {\bf t} & F_0 & F_0\\ \hline F_0 & {\bf t} & {\bf t} & {\bf t} \\
\hline F_l & {\bf t} & {\bf t} & {\bf t} \\ \hline
\end{array} \hspace{1cm} (1\leq l \leq n)$$

\

\noindent Note that both $P^0$ and $I^0$ coincide with the
classical propositional logic over $\{\neg,\rar\}$ with two-valued
matrix semantics. Now we will analyze the strong fibring induced by admissible pairs $(\lambda,\mu)$, with $\lambda: I^n \lra P^k$ and $\mu: P^k \lra I^n$. For that note that, given $I^n$ and $P^k$, the admissible pairs are of the form $({\lambda}_j,{\mu}_i)$ (for $0\leq j \leq k$ and
$0\leq i\leq n$) such that:
\begin{itemize} 
  \item ${\mu}_i({\bf f})=F_i$; ${\lambda}_j({\bf
t})=T_j$
\item ${\mu}_i(T_h)={\bf t}$ and ${\lambda}_j(F_l)={\bf f}$ for $0\leq
h\leq k$ and $0\leq l \leq n$.
\end{itemize} 
\noindent For every fibring pair $(\lambda_j, \mu_i)$, the common signature of $M_{({\lambda}_j,{\mu}_i)}$ is $C_{\circledast} = \{\neg_{{}_{I^n}}, \rar_{{}_{I^n}},\neg_{{}_{P^k}},
\rar_{{}_{P^k}} \}$ and $A_{\circledast} = \{{\bf t}, T_0, T_1,\ldots,T_k,F_0,
F_1,\ldots,F_n,{\bf f}\}$. In addition, $D_{\circledast} = \{{\bf t}, T_0,
T_1,\ldots,T_k\}$, common to all the matrices $M_{(\lambda_j,\mu_i)}$. The operations are given below (the
truth-tables of the negations consider the cases: $i=0$ and
$i>0$; $j=0$ and $j>0$).

\

$$\begin{array}{|c|c|c|c|c|c|c|} \hline \rar_{In}^i & {\bf t} & T_0 & T_h & F_0 & F_l & {\bf f} \\
\hline {\bf t} & {\bf t} & {\bf t} & {\bf t} & F_0 & F_0 & F_0\\
\hline T_0 & {\bf t} & {\bf t} & {\bf t} & F_0 & F_0 & F_0\\
\hline T_h & {\bf t} &
{\bf t} & {\bf t} & F_0 & F_0 & F_0\\ \hline F_0 & {\bf t} & {\bf t} & {\bf t} &
{\bf t} & {\bf t} & {\bf t} \\
\hline F_l & {\bf t} & {\bf t} & {\bf t} & {\bf t} & {\bf t} & {\bf t} \\
\hline {\bf f} & {\bf t} & {\bf t} & {\bf t} & {\bf t} & {\bf t}
& {\bf t}\\ \hline
\end{array} \hspace{1cm}
\begin{array}{|c|c|c|c|c|c|c|} \hline \rar_{Pk}^j & {\bf t} & T_0 & T_h & F_0 & F_l & {\bf f} \\
\hline {\bf t} & T_0 & T_0 & T_0 & {\bf f} & {\bf f} & {\bf f} \\
\hline T_0 & T_0 & T_0 & T_0 & {\bf f} & {\bf f} & {\bf f} \\
\hline T_h & T_0 & T_0 & T_0 & {\bf f} & {\bf f} & {\bf f}\\
\hline F_0
& T_0 & T_0 & T_0 & T_0 & T_0 & T_0\\ \hline F_l & T_0 & T_0 & T_0 & T_0 & T_0 & T_0\\
\hline {\bf f} & T_0 & T_0 & T_0 & T_0 & T_0 & T_0\\ \hline
\end{array}$$

\

$$\begin{array}{|c|c|c|c|c|c|c|} \hline  & {\bf t} & T_0 & T_h & F_0 & F_l & {\bf f} \\
\hline \neg_{In}^0 & F_0 & F_0 & F_0 & {\bf t} & F_{l-1} & {\bf t} \\
\hline  \neg_{In}^i & F_0 & F_0 & F_0 & {\bf t} & F_{l-1} & F_{i-1} \\
\hline \neg_{Pk}^0 & {\bf f} & {\bf f} & T_{h-1} & T_0 & T_0 &
T_0 \\ \hline \neg_{Pk}^j & T_{j-1} & {\bf f} & T_{h-1} & T_0 &
T_0 & T_0 \\ \hline

\end{array}  \hspace{1cm} (1\leq h \leq k; \ 1\leq l \leq n) 
$$

\

\noindent Each matrix  $M_{({\lambda}_j,{\mu}_i)}$ defines a logic that is
simultaneously paraconsistent (w.r.t. $\neg_{Pk}$) and
paracomplete (w.r.t. $\neg_{In}$). These logics can be related with some another paraconsistent and paracomplete (that is, paradefinite) logics already defined in the literature. We shall return to this point in the last section. \hfill $\bsq$
\end{example}

The second example concerns the combination of Belnap and Dunn's logic $FDE$ and D'Ottaviano and da Costa's logic $J3$.

\begin{example} \label{FDE}
The Belnap and Dunn's logic $FDE$ (see {\rm~\cite{bel:77a}}, {\rm~\cite{bel:77b}}, {\rm~\cite{omo:wan:17}}) can be defined by means of the $C_{FDE}$-matrix $M_{FDE} = ({\bf A_{FDE}}, \{ {\bf t}, {\bf b} \})$ with support $A_{FDE} = 
\{{\bf t}, {\bf b}, {\bf n}, {\bf f}\}$ and  connectives $|C_{FDE}| = \{\sim, \wedge, \vee \}$ (with the usual arities) such that the operators are defined as follows:

\

$\begin{array}{|c|c| c| c|c|} \hline & {\bf t} & {\bf b} & {\bf n} & {\bf f} \\
\hline \sim & {\bf f} & {\bf b}  & {\bf n} & {\bf t}   \\
\hline
\end{array}
\hspace{1cm}\begin{array}{|c|c|c|c|c|} \hline \wedge & {\bf t} & {\bf b} & {\bf n} & {\bf f} \\
\hline {\bf t} & {\bf t} & {\bf b}  & {\bf n} & {\bf f}   \\
{\bf b} & {\bf b} & {\bf b}  & {\bf f} & {\bf f}   \\
{\bf n} & {\bf n} & {\bf f}  & {\bf n} & {\bf f}   \\
{\bf f} & {\bf f} & {\bf f}  & {\bf f} & {\bf f}   \\
\hline
\end{array}
\hspace{1cm}
\begin{array}{|c|c|c|c|c|} \hline \vee & {\bf t} & {\bf b} & {\bf n} & {\bf f} \\
\hline {\bf t} & {\bf t} & {\bf t}  & {\bf t} & {\bf t}   \\
{\bf b} & {\bf t} & {\bf b}  & {\bf t} & {\bf b}   \\
{\bf n} & {\bf t} & {\bf t}  & {\bf n} & {\bf n}   \\
{\bf f} & {\bf t} & {\bf b}  & {\bf n} & {\bf f}   \\
\hline
\end{array}$

\

\

\noindent In turn, the 3-valued D'Ottaviano and da Costa's paraconsistent logic $J3$ introduced in {\rm\cite{dot:dac:70}} is defined by means of the $C_{J3}$-matrix $M_{J3} = ({\bf A_{J3}}, \{ 1, \frac{1}{2} \})$ with support $A_{J3} = \{0, 1, \frac{1}{2}\}$, connectives $|C_{J3}| = \{\n, \veebar, \nabla \}$ (where $\n$, $\veebar$ are binary, and $\nabla$ is unary),  with the following interpretation:

\

$\begin{array}{|c|c| c|c |} \hline & 1 & \frac{1}{2} & 0 \\
\hline \n & 0 & \frac{1}{2}  & 1\\
\hline
\end{array}
\hspace{1cm}
\begin{array}{|c|c| c|c |} \hline & 1 & \frac{1}{2} & 0 \\
\hline \nabla & 1 & 1  & 0\\
\hline
\end{array}
\hspace{1cm}
\begin{array}{|c|c|c|c|} \hline \veebar & 1 & \frac{1}{2} & 0 \\
\hline 1 & 1 & 1  & 1  \\
\frac{1}{2} & 1 & \frac{1}{2}  & \frac{1}{2}    \\
0 & 1 & \frac{1}{2}  & 0  \\
\hline
\end{array}
$

\

\noindent For any set $X$ let $\textsf{c}(X)$ be its cardinal. Then, $\textsf{c}(Y^X)=\textsf{c}(Y)^{\textsf{c}(X)}$. Now, let $M_i = ({\bf A_i}, D_i)$ with  support $A_i$ for $i=1,2$ two non-trivial matrices. That is,  $1 \leq \textsf{c}(D_i) < \textsf{c}(A_i)$ (hence, $1 \leq \textsf{c}(A_i \setminus D_i) < \textsf{c}(A_i)$ for $i = 1,2$). Now, each admissible  pair $(\lambda, \mu)$ can be seen as a $4$-uple $(\lambda_a, \lambda_{b}, \mu_{a}, \mu_b)$ such that $\lambda_a: D_1 \to D_2$; $\lambda_{b}: A_1 \setminus D_1 \to A_2 \setminus D_2$; $\mu_a: D_2 \to D_1$; and $\mu_{b}: A_1 \setminus D_1 \to A_2 \setminus D_2$. From this, the number of admissible pairs for the fibring of $M_1$ with $M_2$ is   $Nap(M_1,M_2) = \textsf{c}(D_2)^{\textsf{c}(D_1)} \cdot \textsf{c}(D_1)^{\textsf{c}(D_2)} \cdot \textsf{c}(A_2 \setminus D_2)^{\textsf{c}(D_1 \setminus A_1)} \cdot \textsf{c}(A_2 \setminus D_2)^{\textsf{c}(A_1 \setminus D_1)}$. 

Taking this into account, there are  $32= 4 . 4 . 2 . 1$  admissible pairs for producing fibrings of $FDE$ and $J3$. Each of  these fibrings determines four variants of the Law of Excluded Middle (LEM), two of them being hybrids, namely:

\bc $\n \varphi \vee \varphi$ (I); \, \, \,  ${\sim} \varphi \veebar \varphi$ (II) \ec

\noindent Let us analize the behavior of both versions of (LEM) using different admissible pairs. For instance, consider the admissible pair

\

\noindent $(\lambda_1,\mu_1): 
\left\{
\begin{array}{llll} \lambda_1({\bf t}) = 1; & \lambda_1({\bf b}) = \frac{1}{2}; & \lambda_1({\bf n}) = 0; & \lambda_1({\bf f}) = 0\\ 
\mu_1(1) = {\bf t}; &   \mu_1(\frac{1}{2}) = {\bf t}; &  \mu_1(0) = {\bf n} & \\
\end{array}
\right . .
$

\

\noindent It is routine to check that $\models_{(M_{FDE} \circledast M_{J3})_{(\lambda_1,\mu_1)}} \n \varphi \vee \varphi$.  Moreover, in  Proposition~\ref{propfde-j3}  a more general  result will be stated.

\hfill $\bsq$
\end{example}

\begin{proposition} \label{propfde-j3} {$\models_{(M_{FDE} \circledast M_{J3})_{(\lambda,\mu)}}\n \varphi \vee \varphi$, for  every admissible pair $(\lambda,\mu)$.}
\end{proposition}

\begin{proof} Observe first the following, for $x,y \in A_{FDE}$ and  $z \in A_{J3}$:\\[1mm]
 $(1)$ \ \  $x \vee y  \in D_{FDE}$ \ iff \ either $x  \in D_{FDE}$ or $y  \in D_{FDE}$; and \\[1mm]
 $(2)$ \ \ $z  \notin D_{J3}$ \ implies that \  $\neg z  \in D_{J3}$.\\[1mm]
Now, let $v$ be a valuation over $(M_{FDE} \circledast M_{J3})_{(\lambda,\mu)}$ and let $x:=v(\varphi)$. Then, $v(\neg \varphi)= \neg \ast_{\lambda}(x)$ and so $v(\n \varphi \vee \varphi) = \ast_{\mu}(\neg \ast_{\lambda}(x)) \vee \ast_{\mu}(x)$. From this:\\[1mm]
(i)  Suppose that  $x \in D_{\circledast}$. Then $\ast_{\mu}(x) \in D_{FDE}$ and so  $v(\n \varphi \vee \varphi)  \in D_{FDE}$, by  $(1)$.\\[1mm]
(ii) Suppose that  $x \notin D_{\circledast}$. Then, $\ast_{\lambda}(x)  \notin D_{J3}$ and so $\neg \ast_{\lambda}(x) \in D_{J3}$, by~$(2)$. From this $\ast_{\mu}(\neg \ast_{\lambda}(x))  \in D_{FDE}$, hence  $v(\n \varphi \vee \varphi)  \in D_{FDE}$, by  $(1)$. 
\end{proof}

\

The previous result lies on facts~(1) and~(2).  On the other hand, if we consider the version (II) of (LEM), our analysis should be slightly different, given that the analogous to~$(2)$ does not hold in $FDE$ for $\sim$ (in spite of the analogous to~$(1)$ being true in $J3$ for $\veebar$). Specifically, recalling that ${\sim}{\bf n} = {\bf n}$, it holds that $\not\models_{(M_{FDE}\circledast M_{J3})_{(\lambda_1,\mu_1)}} {\sim} \varphi \veebar \varphi$, for the pair $(\lambda_1,\mu_1)$ mentioned in Example~\ref{FDE}. Indeed, it is enough proving this fact for $\varphi$ being a propositional variable, namely $p$. Thus, given a valuation $v$ over $(M_{FDE} \circledast M_{J3})_{(\lambda_1,\mu_1)}$ with $v(p) = {\bf n}$, we have that $v({\sim} p \veebar p) = {\sim} {\bf n} \veebar {\bf n} = {\bf n} \veebar {\bf n} = 0 \notin D_{\circledast}$. Moreover, by the same argument it holds that, for any admissible par $(\lambda,\mu)$ such that $\mu(0)={\bf n}$, $v({\sim} p \veebar p) =  0 \notin D_{\circledast}$ for every valuation $v$ over $(M_{FDE} \circledast M_{J3})_{(\lambda,\mu)}$. On the other hand, if  $(\lambda,\mu)$ is an admissible pair such that  $\mu(0)={\bf f}$ then it is immediate to see that $v({\sim} p \veebar p) \in D_{\circledast}$ for every  valuation $v$ over $(M_{FDE} \circledast M_{J3})_{(\lambda,\mu)}$. That is:

\begin{proposition}{For every admissible pair $(\lambda,\mu)$, it holds the following:  $\models_{(M_{FDE} \circledast M_{J3})_{(\lambda,\mu)}} {\sim}\varphi \vee \varphi$ \ iff \  $\mu(0)={\bf f}$.}
\end{proposition}

\

Finally, we will  analyze the definability of connectives with relation to fibring by functions. Recall that, given a $C$-algebra ${\bf A}=(A, C^{\bf A})$, any formula (or {\em term}, in the terminology of universal algebra) $\varphi(p_1,\ldots,p_n)$ depending  on the variables $p_1,\ldots,p_n$ defines a new $n$-ary connective $c_\varphi$ such that $c_\varphi^{\bf A}(\vec a)=\varphi^{\bf A}(\vec a)$ for every $\vec a \in A^n$. Such function defined from a term is called a {\em term function}. For instance, in the algebra ${\bf A}_{P^1}$ underlying the matrix of the 3-valued logic $P^1$ (recall Example~\ref{P1-1}) the term $\varphi(p)=\neg_{P^1}(\neg_{P^1} p \to_{P^1} p)$ defines a truth-function $f_\sim:=\varphi^{{\bf A}_{P^1}}$ which behaves as a classical negation, that is: $f_{\sim}(T)=f_{\sim}(T_1)=F$ and $f_{\sim}(F)=T$. By definition, the function $f_\sim$ is a term function.

Given a formula $\varphi$ as above, let $C_\varphi$ be the signature obtained from $C$ by adding a new $n$-ary connective $c_\varphi$, and let ${\bf A}_\varphi=(A, {C_\varphi}^{{\bf A}_\varphi})$ be the expansion of ${\bf A}$ to $C_\varphi$ such that $c_\varphi^{{\bf A}_\varphi}:=\varphi^{\bf A}$.
Given a $C$-matrix $M=({\bf A}, D)$  based on ${\bf A}$ let  $M_\varphi=({\bf A}_\varphi, D)$ be the $C_\varphi$-matrix based on ${\bf A}_\varphi$ obtained from $M$. Clearly, the associated matrix logics 
 $\lc_M= \lan C,\models_{M} \ran$ and  $\lc_{M_\varphi}= \lan C_\varphi,\models_{M_\varphi} \ran$ coincide up to language. To be more specific, suppose that  ${\bf A}=(A, C^{\bf A})$ and ${\bf A'}=(A, {C'}^{\bf A'})$ are algebras over $C$ and $C'$, respectively,  with the same universe $A$ such that they have the same set of term functions (which is called, in the terminology of universal algebra, the {\em clone} of the algebras, see for instance~\cite{mck:etal:87}). Then, any matrices $M=({\bf A}, D)$ and $M'=({\bf A'}, D)$  based on ${\bf A}$ and ${\bf A'}$, respectively, generate the same logic up to language (each of such matrix logic can be considered as an  specific {\em presentation} of the logic generated by $M$).  For instance, classical propositional logic $CPL$ can be presented as 2-valued matrices over signatures $|C_1|=\{\neg_{CPL},\land_{CPL}\}$,  $|C_2|=\{\neg_{CPL},\lor_{CPL}\}$, $|C_3|=\{\bot_{CPL},\to_{CPL}\}$ or $|C_4|=\{\uparrow_{CPL}\}$ (where $\uparrow_{CPL}$ denotes the Sheffer stroke), among others.

In Corollary~\ref{stable} it will be shown that fibring by functions is a mechanism for combining matrix logics which is stable under presentation of the given matrix logics.

\begin{proposition} \label{expan-fun}
Let  ${\bf A}_i=(A, C_i^{{\bf A}_i})$ be a $C_i$-algebra with universe $A$ for $i=1,2$ such that $C_1 \sqsubseteq C_2$ and, for every $n \geq 0$ and $c \in C_2^n$: if $c \in C_1^n$ then $c^{{\bf A}_2}=c^{{\bf A}_1}$; and, if  $c \not\in C_1^n$, there exists  a non atomic formula  $\varphi_c(p_1,\ldots,p_n)$ of $L(C_1)$ depending  on the variables $p_1,\ldots,p_n$ such that  $c^{{\bf A}_2}=\varphi_c^{{\bf A}_1}$.
For $i=1,2$ let $M_i=({\bf A}_i, D)$ be a  $C_i$-matrix with universe $D$ based on ${\bf A}_i$. Let ${\bf A'}=(A', {C'}^{\bf A'})$ be a $C'$-algebra and let $M'=({\bf A'}, D')$ be a $C'$-matrix based on ${\bf A'}$. Given functions $\lambda: A \to A'$ and $\mu:A' \to A$  let $\bar{M}^i_{(\lambda,\mu)} := (M_i\circledast M')_{(\lambda,\mu)}$ for $i=1,2$.\footnote{Observe that both matrices $\bar{M}^1_{(\lambda,\mu)}$ and $\bar{M}^2_{(\lambda,\mu)}$ have the same universe and set of designated values, but different signatures.}
Then, the matrix logics  $\lan \bar{C}^1_{\circledast}, \models_{\bar{M}^1_{(\lambda,\mu)}}\ran$ and $\lan \bar{C}^2_{\circledast}, \models_{\tilde{M}^2_{(\lambda,\mu)}}\ran$ coincide up to language, that is: they are different presentations of the same logic.
\end{proposition}
\begin{proof} 
It is an immediate consequence of the definitions and the observations made right after Example~\ref{InPk}. Indeed, by hypothesis, ${\bf A}_1$ and ${\bf A}_2$  have the same set of term functions and so $M_1$ and $M_2$ generate the same matrix logic up to language. Now, let $\bar{M}^i_{(\lambda,\mu)} = ({\bf B^{\textit{i}}_{\circledast}}^{(\lambda,\mu)},\bar{D}_{\circledast})$. By construction, ${\bf B^{\textit{i}}_{\circledast}}^{(\lambda,\mu)}$ is a  $\bar{C^i}_{\circledast}$-algebra with universe $A\uplus A'$ for $i=1,2$ such that $\bar{C^1}_{\circledast} \sqsubseteq \bar{C^2}_{\circledast}$ and, for every $n \geq 0$ and $c \in (\bar{C^2}_{\circledast})^n$: if $c \in (\bar{C^1}_{\circledast})^n$ then $c^{{\bf B^{\textrm{2}}_{\circledast}}^{(\lambda,\mu)}}=c^{{\bf B^{\textrm{1}}_{\circledast}}^{(\lambda,\mu)}}$; and, if  $c \not\in (\bar{C^1}_{\circledast})^n$, there exists  a non atomic formula  $\psi_c(p_1,\ldots,p_n):=inj_1(\varphi_c(p_1,\ldots,p_n))$ of $L(\bar{C^1}_{\circledast})$ depending  on the variables $p_1,\ldots,p_n$ such that  $c^{{\bf B^{\textrm{2}}_{\circledast}}^{(\lambda,\mu)}}=\psi_c^{{\bf B^{\textrm{1}}_{\circledast}}^{(\lambda,\mu)}}$. Here, $inj_1$ is the canonical injection from $C_1$ to $\bar{C^1}_{\circledast}$ associated to the disjoint union $\bar{C^1}_{\circledast}=C_1 \uplus C'$. This implies that ${\bf B^{\textrm{1}}_{\circledast}}^{(\lambda,\mu)}$ and ${\bf B^{\textrm{2}}_{\circledast}}^{(\lambda,\mu)}$  have the same set of term functions and so we conclude that  $\bar{M}^1_{(\lambda,\mu)}$ and $\bar{M}^2_{(\lambda,\mu)}$ generate the same matrix logic up to language.
\end{proof}
 
\noindent  From this, we obtain the general case:

\begin{proposition} \label{expan-fun2}
Let  ${\bf A}_i=(A, C_i^{{\bf A}_i})$ be a $C_i$-algebra with universe $A$ for $i=1,2$ such that, for every $n \geq 0$ and $c \in C_1^n \cup C_2^n$: if $c \in C_1^n \cap C_2^n$ then $c^{{\bf A}_1}=c^{{\bf A}_2}$; and, if  $c \in C_j^n\setminus C_i^n$, there exists  a non atomic formula  $\varphi_c(p_1,\ldots,p_n)$ of $L(C_i)$ depending  on the variables $p_1,\ldots,p_n$ such that  $c^{{\bf A}_j}=\varphi_c^{{\bf A}_i}$, for $i\neq j$.
For $i=1,2$ let $M_i=({\bf A}_i, D)$ be a  $C_i$-matrix with universe $D$ based on ${\bf A}_i$. Now, let ${\bf A'}=(A', {C'}^{\bf A'})$ be a $C'$-algebra and let $M'=({\bf A'}, D')$ be a $C'$-matrix based on ${\bf A'}$. Given functions $\lambda: A \to A'$ and $\mu:A' \to A$ let $\bar{M}^i_{(\lambda,\mu)} := (M_i\circledast M')_{(\lambda,\mu)}$ for $i=1,2$.
Then, the matrix logics  $\lan \bar{C}^1_{\circledast}, \models_{\bar{M}^1_{(\lambda,\mu)}}\ran$ and 
$\lan \bar{C}^2_{\circledast}, \models_{\tilde{M}^2_{(\lambda,\mu)}}\ran$ coincide up to language, that is: 
they are different presentations of the same logic.
\end{proposition} 
\begin{proof}
Let $C=C_1 \cup C_2$ and let ${\bf A}=(A, C^{{\bf A}})$ be the $C$-algebra such that $c^{{\bf A}}=c^{{\bf A}_1}$ if $c \in |C_1|$, and $c^{{\bf A}}=c^{{\bf A}_2}$ if $c \in |C_2|$. By hypothesis, $c^{{\bf A}}$ is well-defined. Now, given  $M_i=({\bf A}_i, D)$ ($i=1,2$), $M'=({\bf A'}, D')$ and $(\lambda,\mu)$, let $\bar{M}^i_{(\lambda,\mu)} := (M_i\circledast M')_{(\lambda,\mu)} = ({\bf B^{\textit{i}}_{\circledast}}^{(\lambda,\mu)},\bar{D}_{\circledast})$ $(i=1,2$) and let  $\bar{M}_{(\lambda,\mu)} := (M\circledast M')_{(\lambda,\mu)} = ({\bf B_{\circledast}}^{(\lambda,\mu)},\bar{D}_{\circledast})$, where $M=({\bf A}, D)$.  By Proposition~\ref{expan-fun},  the matrix logics  $\lan \bar{C}^i_{\circledast}, \models_{\bar{M}^i_{(\lambda,\mu)}}\ran$ and $\lan \bar{C}_{\circledast}, \models_{\tilde{M}_{(\lambda,\mu)}}\ran$ coincide up to language, for $i=1,2$. This means that  $\lan \bar{C}^1_{\circledast}, \models_{\bar{M}^1_{(\lambda,\mu)}}\ran$ and  $\lan \bar{C}^2_{\circledast}, \models_{\bar{M}^2_{(\lambda,\mu)}}\ran$  coincide up to language, that is: they are different presentations of the same logic.
\end{proof}

\begin{corollary} \label{stable}
Let $M_1$, $M_2$ and $M'$ be as in Proposition~\ref{expan-fun2}, and let $\lc_1$, $\lc_2$ and $\lc'$ be the corresponding logic matrices. Then, the fibring by functions $[\lc_1 \circledast \lc']_{(\lambda,\mu)}$ and $[\lc_2 \circledast \lc']_{(\lambda,\mu)}$ coincide up to language, that is: they are different presentations of the same logic.
\end{corollary}
\begin{proof}
It is an immediate consequence of Theorem~\ref{fibring-es-matricial} and Proposition~\ref{expan-fun2}.
\end{proof}

\noindent The meaning of the last corollary is as follows: let $\lc_1=\lan C_i, \models_{M_i}\ran$ be two matrix logics where $M_i=({\bf A}_i, D)$, for $i=1,2$, such that the underlying algebras ${\bf A}_1$ and ${\bf A}_2$ have the same clone, that is, both algebras can generate the same term functions (this means that $\lc_1$ and $\lc_2$ are the same logic, up to language). Moreover, assume the following: if one of the logics is able to express the $ith$-projection $\pi_i^n(\vec a)=a_i$ by means of a non-atomic formula $\varphi_i^n(\vec p)$ of its language, so is the other logic. Then,  the logics  $[\lc_1 \circledast \lc']_{(\lambda,\mu)}$ and $[\lc_2 \circledast \lc']_{(\lambda,\mu)}$ coincide up to language, for any matrix logic $\lc'$.

\begin{remarks} (1)
The requirement in Propositions~\ref{expan-fun} and~\ref{expan-fun2}  for the formula $\varphi_c$ not to be a propositional variable is essential. Indeed, consider for instance the case of Proposition~\ref{expan-fun2}. Suppose that, for some $c \in C_2^n\setminus C_1^n$, $c^{{\bf A}_2}$ is the $ith$-projection $\pi_i^n$, but the only way to express such truth-function in  ${\bf A}_1$ is by means of the formula $\varphi_c(\vec p)=p_i$. Let $M'$ such that $A \cap A'=\emptyset$, and let $\vec a \in (A \uplus A')^n$ such that $a_i \in A'$.  To lighten notation, let ${\bf B}_i := {\bf B^{\textit{i}}_{\circledast}}^{(\lambda,\mu)}$ for $i=1,2$ be the algebras underlying the fibred matrices $(M_i \circledast M')_{(\lambda,\mu)}$.  Then, $c^{{\bf B}_2}(\vec a)=\ast_{\mu}(a_i)=\mu(a_i) \in A$. On the other hand,  $\varphi_c^{{\bf B}_1}(\vec a)=p_i^{{\bf B}_1}(\vec a)=a_i \in A'$. This implies that $c^{{\bf B}_2}(\vec a) \neq \varphi_c^{{\bf B}_1}(\vec a)$.\\
(2) It is worth noting that, while the term functions of $\lc_i$ are the same as the ones generated in $[\lc_1 \circledast \lc_1]_{(\lambda,\mu)}$ within the signature $C_i$ ($i=1,2$), the latter matrix logic can generate new term functions by means of hybrid formulas.
\end{remarks}

\section{Constrained Fibring by Functions} \label{fibring-sharing}

In an intuitive way, we can say that two logics $\lc_i = \lan C_i, \vdash_i \ran$ $(i=1,2)$  {\em share a connective $c$} when $c$ belong to both signatures $C_i$, keeping  {\em the same logical meaning} in both logics (we will formalize this idea later). The purpose of this section is, precisely, to find some conditions that allow us to combine logics sharing connectives. 
Now, since the formalism of the previous sections requires to work with disjoint signatures, we must to make the combination  {\em identifying two different connectives as being  the same} (at least with respect to the definition of the relation $\models_{M_{(\lambda,\mu)}}$). 
This kind of combinations can be understood as fibring constrained to the  additional condition $c_1 = c_2$ (with $c_i \in C_i^k$, $i = 1,2$), which justifies the name of this section. Thus, from now on, in this article the expression  {\em $[\lc_1 \circledast \lc_2]_{(\lambda,\mu)}$ constrained to $c_1 = c_2$} just means that $c_1$ and $c_2$ are identified as being the same connective, in the context of $[\lc_1 \circledast \lc_2]_{(\lambda,\mu)}$.

Because all the expressed above, in this section we will be focused on the following problem: What conditions should verify a fibring pair $(\lambda,\mu)$ if we wish to consider two connectives $c_1$ and $c_2$ of different signatures as the same one? Of course, the first question to solve is to find a good formalism for the notion of the identification of connectives. A possible approach to this intuitive idea will be developed in the sequel. 

\begin{definition} \label{formulas-asociadas}  Given a signature $C$,  and given $c_1, c_2 \in {C^k}$, we say that the formulas $\varphi_1$ and $\varphi_2$ are 
 {\em closely $(c_1,c_2)$-associated} iff
some of the three conditions is satisfied:
\begin{itemize}
\item[$(i)$] $\varphi_1 = \varphi_2$ 

\item[$(ii)$] $\varphi_2$ results from the replacement, in $\varphi_1$, of exactly one occurrence of $c_1$ by $c_2$.  

\item[$(iii)$] $\varphi_1$ results from the replacement, in $\varphi_2$, of exactly one occurrence of $c_2$ by $c_1$. 
\end{itemize}

\noindent Now, the formulas $\varphi_1$ and $\varphi_2$ are said to be {\em $(c_1,c_2)$-associated} if there is a finite sequence $(\psi_0,\dots,\psi_r)$ of formulas in 
$L(C)$ such that $\psi_0 = \varphi_1$, $\psi_r = \varphi_2$ and, for every $j = 1,\dots,r$, $\psi_{j-1}$ is closely $(c_1,c_2)$-associated to $\psi_{j}$
\end{definition}

\noindent
For instance, $p \vee q$ and $p \wedge q$ are closely $(\vee,\wedge)$-associated formulas.

\begin{remark}\label{asociadas-relacion-equivalencia} \rm{The notion of $(c_1,c_2)$-associated formulas defines an equivalence 
relation $\equiv_{(c_1,c_2)}$ in the obvious way.
So, $\varphi_1 \equiv_{(c_1,c_2)} \varphi_2$ indicates that $\varphi_1$ is $(c_1,c_2)$-associated to $\varphi_2$. On the other hand, if $\varphi_1 \in \V$, then $\varphi_1 \equiv_{(c_1,c_2)} \varphi_2$ if and only if $\varphi_1$ =$\varphi_2$.}
\end{remark}

\noindent
The next example will clarify the previous definition (in it, we have changed a little the formalism of \cite{haj:98}):

\begin{example} \label{ejemplo-formulas-asociadas}
{Let $C_{{}{\textrm{\L}}_{\aleph_1}}$ be the signature of \L ukasiewicz{'}s logic ${\textrm{\L}}_{\aleph_1}$, being $C_{{}{\textrm{\L}}_{\aleph_1}}$ := 
$\{\&,  \wedge, \vee, \oplus, \rar, \neg \}$. This logic is defined by the matrix $M_{{}{\textrm{\L}}_{\aleph_1}}$,  whose support set is $[0,1]\su \mathbb{R}$. Here, $\&$ and $\wedge$ can be intended as the \em{strong conjunction} and the \em{inf-conjunction}, respectively. In a similar way, $\oplus$ and $\vee$ are the \em{strong disjunction} and the \em{sup-disjunction}. Anyway, the meaning/behavior of the operations defined in $M_{{}{\textrm{\L}}_{\aleph_1}}$ is irrelevant
in the context of this example. Consider now the following formulas:

 $\psi_1 = \n (p \wedge q) \, \& \, r$

 $\psi_2 = \n (p \, \& \, q) \, \& \, r$

 $\psi_3 = \n (p \, \& \, q) \wedge r$

 $\varphi = \n (p \oplus q) \oplus r$

\noindent Then, $\psi_1 \equiv_{(\&,\wedge)} \psi_2$, $\psi_2 \equiv_{(\&,\wedge)} \psi_3$, and $\psi_1 \equiv_{(\&,\wedge)} \psi_3$. 
Besides that, $\psi_1$ and $\psi_2$ are closely $(\&,\wedge)$-associated, but $\psi_1$ and $\psi_3$ are not. Finally, note that $\varphi$ is $(\oplus, \&)$-associated to $\psi_2$. This last example illustrate the fact that the notion of $(c_1,c_2)$-association is merely linguistic, without any reference to the logical meaning of the involved connectives.  \hfill{$\bsq$}
}
\end{example}

\begin{definition}\label{definicion-sharing-abstracta} {Let $\lc = \lan C, \models_M \ran$ be a matrix logic, where $M = ({\bf A},D)$ has support $A$. We say that $\lc$  {\em identifies $c_1$ with $c_2$} ($c_1$, $c_2 \in C^k$) if, for every pair of $(c_1, c_2)$-associated formulas $\varphi_1$, $\varphi_2 \in L(C)$ depending of $m$ variables, for every $\vec{a} =(a_1,\dots,a_m) \in A^m$, it holds that:
 $$\varphi_1^{\bf A}(\vec{a}) \in D\textit{ iff } \varphi_2^{\bf A}(\vec{a}) \in D$$}
\end{definition}

\begin{remark} \label{equivalencias-funciones-verdad}{Consider the equivalence relation $\simeq_D \su L(C) \times L(C)$ defined as follows: $\varphi_1 \simeq_D \varphi_2$ (where 
$\varphi_1$ and $\varphi_2$ depend at most of $m$ variables)\footnote{If $\varphi_i$ depends on $m_i$ variables then $\varphi_i = \varphi_i(p_1,\dots,p_m)$ (for $i = 1,2$), where $m = max(\{m_1,m_2\})$.}  
if and only if, for every $\vec{a} \in A^m$, $\varphi_1(\vec{a}) \in D$ iff $\varphi_2(\vec{a}) \in D$. It is easy to see that $L(C)$ identifies $c_1$ with $c_2$ if and only if, for every pair $(\varphi_1,\varphi_2) \in L(C) \times L(C)$, 
$\varphi_1 \equiv_{(c_1,c_2)} \varphi_2$ implies that $\varphi_1 \simeq_D \varphi_2$.}
\end{remark}

\noindent
It is worth noting that, if two connectives $c_1$, $c_2 \in C$ can be identified in the context of a matrix logic $\lc = \lan C, \models_M \ran$, then one of them becomes redundant. So, it would be possible to reduce $C$ to obtain a signature $C{'}$, by the elimination of $c_1$ (or $c_2$). And, therefore, it would be possible to define a logic  $\lc{'}$, equivalent to $\lc$ (but defined by means of the reduced signature $C{'}$). Even when the process of elimination of redundant connectives is not treated in this paper, we will refer to it informally later. On the other hand realize that the previous definition is, actually, a characterization of the intuitive meaning of the identification of connectives, as the next result (of easy proof) shows:

\begin{proposition} Let $\lc = \lan C, \models_{M} \ran$ be a matrix logic, with $M = ({\bf A}, D)$, and suppose that $c_1$, $c_2 \in C^k$. Then, the following conditions are equivalent:
\begin{itemize}
\item [(i)] $\lc$ identifies $c_1$ with $c_2$, cf. Definition \ref{definicion-sharing-abstracta}.
\item [(ii)] For every $M$-valuation $v$, for every $\varphi_1$, $\varphi_2 \in L(C)$, if $\varphi_1 \equiv_{(c_1,c_2)}\varphi_2$, then $v(\varphi_1) \in D$ if and only if 
$v(\varphi_2) \in D$.
  \item [(iii)] For every $\varphi_1$, $\varphi_2 \in L(C)$, $\varphi_1 \equiv_{(c_1,c_2)}\varphi_2$ implies that $\varphi_1 \models_{M} \varphi_2$ and $\varphi_2 \models_{M} \varphi_1$.
\end{itemize}
\end{proposition}

\begin{definition}\label{definicion-sharing-importante}{
For $i = 1,2$, consider the signatures $C_i$, two fixed connectives $c_i \in C_{i}^k$, and two $C_i$-matrices 
$M_i = ({\bf A}_i, D_i)$ with support $A_i$ ($i = 1,2$). We say that the fibring pair $(\lambda,\mu)$  {\em identifies $c_1$ with $c_2$} if the logic $[\lc_1 \circledast \lc_2]_{(\lambda,\mu)}$ (as in Definition~\ref{matriz-lambda-mu}) identifies $c_1$ with $c_2$, according to Definition~\ref{definicion-sharing-abstracta}.
When the pair $(\lambda,\mu)$ identifies 
$c_1$ with $c_2$ then $[\lc_1 \circledast \lc_2]_{(\lambda,\mu)}$ will be called  {\em the fibring by functions $(\lambda,\mu)$ constrained to $c_1 = c_2$}, and it will be denoted by $[\lc_1 \circledast \lc_2]_{(\lambda,\mu)}^{c_1 = c_2}$.
The same convention holds for $[\lc_1 \odot \lc_2]_{(\lambda,\mu)}^{c_1 = c_2}$.}
\end{definition}


\begin{notation} \label{fragmentos-matrices}
If $M = ({\bf A},D)$, with support $A$, is a $C$-matrix and $c \in C$, we denote by ${\bf A}/c$ the  $\{c\}$-fragment of ${\bf A}$. 
That is, ${\bf A}/c$ is an algebra whose only truth-function $c$ is the inherited from ${\bf A}$ (note, however that the supports $A$ and $A /c$ are the same, say $A$). In addition, $M /c$ denotes the $\{c\}$-matrix $M /c:= ({\bf A}/c, D)$, and it will called  {\em the $c$-fragment of $M$}.
\end{notation}

\begin{theorem}\label{sharing-fundamental-1} Given $C_i$, 
$c_i \in C_{i}^{k}$, and $M_i = ({\bf A}_i, D_i)$ ($i = 1,2$) as in Definition \ref{definicion-sharing-importante}. If the following conditions are verified: 

\noindent $(a)$ $(\lambda,\mu)$ is an admissible pair.

\noindent $(b)$ $\lambda: A_1 \longrightarrow A_2$ is an isomorphism between the respective ${\bf A}_i / c_i$-fragments.

\noindent $(c)$ $\mu = {\lambda}^{-1}$

\noindent Then  $(\lambda,\mu)$ identifies $c_1$ with $c_2$.
\end{theorem}

\noindent That is,  $(\lambda,\mu)$ identifies $c_1$ with $c_2$ if $\lambda$ is a  {\em bijective strict homomorphism} (or, shortly, a matrix isomorphism, cf. \cite{cze:01}) from ${M_1}/{c_1}$ to ${M_2}/{c_2}$, being 
$\mu = {\lambda}^{-1}$. Note that if the conditions of this theorem are valid, then the obtained fibring will be aditionally a strong and conservative one, cf 
Proposition \ref{extendcons}.  We will make use of the functions $\ast_{\lambda}$ and $\ast_{\mu}$ given in Definition \ref{funciones-lambda-mu}, and the ``truth-tables'' induced by that functions, in Definition \ref{matriz-lambda-mu}. First of all, we can demonstrate the following technical result:

\begin{proposition} \label{sharing-tecnico-1} Consider the $C_i$-matrices $M_i = ({\bf A}_i,D_i)$, $(\lambda,\mu)$ a fibring pair, and the matrix $M_{(\lambda,\mu)} = ({\bf A_{\circledast}^{(\lambda,\mu)}},D_{\circledast})$, defined as usual. Then:
\begin{itemize}
  \item[(a)]For every $a \in A_1$, $\ast_{\lambda}(a) = \lambda (\ast_{\mu}(a))$. 
  \item[(b)] If $\lambda$ is a bijection and $\lambda^{-1} = \mu$, then  $\ast_{\lambda}(a) = \lambda (\ast_{\mu}(a))$ for every $a \in A_{\circledast}$.
\item[(c)]For every $a \in A_2$, $\ast_{\mu}(a) = \mu (\ast_{\lambda}(a))$. 
  \item[(d)] If $\lambda$ is a bijection and $\lambda^{-1} = \mu$, then $\ast_{\mu}(a) = \mu (\ast_{\lambda}(a))$ for every $a \in A_{\circledast}$.
\end{itemize}
\end{proposition}
\begin{proof}
For $(a)$: if $a \in A_1$, we have that $\ast_{\mu}(a) = a$. So, $\lambda(\ast_{\mu}(a)) = \lambda(a)$ =  $\ast_{\lambda}(a)$ (since $a\in A_1$, again). For $(b)$: we just must consider when $a \in A_2$. In this case, $\ast_{\lambda}(a) = a$. On the other hand, $\ast_{\mu}(a) = \mu(a)$ and so, $\lambda(\ast_{\mu}(a)) = \lambda(\mu(a)) = a$, by hypothesis. The cases $(c)$ and $(d)$ are similar. 
\end{proof}

\begin{proposition} \label{sharing-tecnico-2} Suppose $c_1 \in C_1^k$ and $c_2 \in C_2^k$, and let $(\lambda,\mu)$ be a fibring pair satisfying:
    
\noindent $(i)$ $\lambda: {\bf A_1} /c_1 \longrightarrow {\bf A_2} / c_2$  is an isomorphism. 

\noindent $(ii)$ $\mu = {\lambda}^{-1}$.

\noindent Then, for every pair of formulas $\psi_1$, $\psi_2  \in L(C_{\circledast})\setminus \V$ depending of $m$ atomic formulas, such that 
$\psi_1 \equiv_{(c_1,c_2)} \psi_2$, for every $\vec{a} = (a_1,\dots,a_m) 
\in {(A_{\circledast})}^m$, it holds:

\noindent $(\star.1)$: $\ast_{\lambda}(\psi_1^{\bf A_{\circledast}^{(\lambda,\mu)}}(\vec{a})) = \ast_{\lambda}(\psi_2^{\bf A_{\circledast}^{(\lambda,\mu)}}(\vec{a}))$ 

\noindent $(\star.2)$: $\ast_{\mu}(\psi_1^{\bf A_{\circledast}^{(\lambda,\mu)}}(\vec{a})) = \ast_{\mu}(\psi_2^{\bf A_{\circledast}^{(\lambda,\mu)}}(\vec{a}))$.
\end{proposition}
\begin{proof}
Our claim will be proved by induction on $n$, the number of connectives appearing in $\psi_1(\vec{p})$, with $\vec{p} = (p_1,\dots,{p}_m)$, with $n \geq 1$\footnote{We do not consider $n = 0$ because, in this case, $\psi_1$  = $\psi_2 \in \V$ (because of Remark~\ref{asociadas-relacion-equivalencia}), and then $(\star.1)$ and $(\star.2)$ are trivially valid. So, the relevant fact of the proof of this proposition just appear when $n \geq 1$.} (note that these connectives belong to $C_{\circledast}$). To simplify notation we will write $\psi(\vec a)$ and $d(\vec x)$ instead of  $\psi^{\bf A_{\circledast}^{(\lambda,\mu)}}(\vec a)$ and $d^{\bf A_{\circledast}^{(\lambda,\mu)}}(\vec{x})$, respectively, for a formula $\psi$ and a $m$-ary connective $d$.

Suppose that $n = 1$. Then, $\psi_1 = d(\vec{p})$, with $d \in C_{\circledast}$. If $c_1 \neq d \neq c_2$, then $\psi_1 = \psi_2$, and our claim is valid. If $d = c_1$, then $\psi_1(\vec{a}) = c_1(\vec{a})$. 
We must just analyze when $\psi_2(\vec{a}) = c_2(\vec{a})$.  
Since $\psi_1(\vec{a}) \in A_1$, and using $(i)$, we have that $\ast_{\lambda}(\psi_1(\vec{a})) = \lambda(c_1(a_1,\dots,a_k)) = \lambda(c_1(\ast_{\mu}(a_1),\dots,\ast_{\mu}(a_k))) = c_2(\lambda(\ast_{\mu}(a_1)),\dots,\lambda(\ast_{\mu}(a_k))) = c_2(\ast_{\lambda}(a_1),\dots,\ast_{\lambda}(a_k)) = c_2(\vec{a})$ (by the previous proposition and Definition \ref{matriz-lambda-mu}). Now, since $c_2 \in C_2^k$, $c_2(\vec{a}) = \ast_{\lambda}(c_2(\vec{a}))$ = 
$\ast_{\lambda}(\psi_2(\vec{a}))$. If $d = c_2$, the reasoning is similar. So, $(\star .1)$ is valid.

Moreover, $(\star . 2)$ can be proved adapting the same patterns applied above, and using the additional hypothesis $(ii)$. 

Now suppose that $(\star.1)$ and $(\star.2)$ are valid for every formula with $t$ connectives ($t < n$) and let us prove that it is valid for every formula $\psi_1$ with $n$ connectives. Since $\psi_1 \equiv_{(c_1,c_2)} \psi_2$, we can consider that $\psi_1 = d(\theta_1^1,\dots,\theta_p^1)$ and $\psi_2 = d{'}(\theta_1^2,\dots,\theta_p^2)$, with $d$, $d{'} \in C_{\circledast}^p$, and with $\theta_i^1 = \theta_i^1(\vec{x})$, $\theta_i^2 = \theta_i^2(\vec{x})$ and
$\theta_i^1 \equiv_{(c_1,c_2)}\theta_i^2$, for every $i=1,\dots,k$. We have the following cases:

\noindent {\bf Case 1}: $d = c_1$ (and $p = k$). We must analyze when $d{'} = c_2$ and when $d{'} = c_1$. In the first 
case, $\psi_2 = c_2(\theta_1^2,\dots,\theta_k^2)$. To prove $(\star .1)$ we have:
$\ast_{\lambda}(\psi_1(\vec{a}))$ = 
$\ast_{\lambda} (c_1 (\theta_1^1(\vec{a}),\dots,\theta_k^1(\vec{a}))$ = 
$\lambda(c_1(\theta_1^1(\vec{a}),\dots,\theta_k^2(\vec{a})))$, which (by Definition \ref{matriz-lambda-mu}) is $\lambda(c_1(\ast_{\mu}(\theta_1^1(\vec{a})),\dots,\ast_{\mu}(\theta_k^1(\vec{a}))))$ =  
$c_2(\lambda(\ast_{\mu}(\theta_1^1(\vec{a}))),\dots,\lambda(\ast_{\mu}(\theta_k^1(\vec{a}))))$ =
$c_2(\ast_{\lambda}(\theta_1^1(\vec{a})),\dots,\ast_{\lambda}(\theta_k^1(\vec{a})))$ (using $(i)$ and Proposition \ref{sharing-tecnico-1}). Applying Induction Hypothesis (case $(\star.1)$) and 
Definition \ref{matriz-lambda-mu}, we have then $\ast_{\lambda}(\psi_1(\vec{a}))$ =  $c_2(\ast_{\lambda}(\theta_1^2(\vec{a})),\dots,\ast_{\lambda}(\theta_k^2(\vec{a})))$ =
$c_2(\theta_1^2(\vec{a}),\dots,\theta_k^2(\vec{a}))$ = 
$\ast_{\lambda} (c_2(\theta_1^2(\vec{a}),\dots,\theta_k^2(\vec{a})))$ = 
$\ast_{\lambda}(\psi_2(\vec{a}))$, because $c_2 \in C_2^k$.  
To prove  $(\star.2)$ proceed in the same way (applying $(ii)$, and Induction Hypothesis, case $(\star.2)$, here). 

If $d{'} = c_1$, our claim is proved by an adaptation of the explanation above.

\noindent {\bf Case 2}: $d = c_2$. Similar to case 1.

\noindent {\bf Case 3}: $d \neq c_1$, $d \neq c_2$. So, $\psi_1 = d(\theta_1^1,\dots,\theta_p^1)$, and $\psi_2 = d(\theta_1^2,\dots,\theta_p^2)$. 
If $d \in C_1^p$ then by Induction Hypothesis (case $(\star.2)$), 
$\ast_{\mu}(\theta_i^1(\vec{a})) = \ast_{\mu}(\theta_i^2(\vec{a}))$, for $i = 1,\dots,p$. Then,
$d(\ast_{\mu}(\theta_1^1(\vec{a})),\dots,\ast_{\mu}(\theta_p^1(\vec{a}))) = d(\ast_{\mu}(\theta_1^2(\vec{a})),\dots,\ast_{\mu}(\theta_p^2(\vec{a})))$, and so   
$d(\theta_1^1(\vec{a}),\dots,\theta_p^1(\vec{a})) = d(\theta_1^2(\vec{a}),\dots,\theta_p^2(\vec{a}))$ (because $d \in C_1^p$). That is, $\psi_1(\vec{a})$ = 
$\psi_2(\vec{a})$, and then $(\star.1)$ and $(\star.2)$ are both valid. Finally, if $d \in C_2^p$, proceed in the same way, applying Induction Hypothesis (case $(\star.1)$, here). This concludes the proof.
\end{proof}

\noindent Note that the admissibility condition is not used in the previous result. Conditions $(i)$ and $(ii)$ of this proposition correspond to conditions $(b)$ and $(c)$ of Theorem \ref{sharing-fundamental-1}, respectively.

Theorem \ref{sharing-fundamental-1} can be demonstrated now:

\

{\noindent {\bf Proof of Theorem \ref{sharing-fundamental-1}:}
Let $\psi_1$, $\psi_2 \in L(C_{\circledast})$ be such that $\psi_1 \equiv_{(c_1,c_2)} \psi_2$. When $\psi_1 \in \V$ our claim is trivial (because $\psi_1 = \psi_2$). Now, 
suppose that $\psi_1 = d(\theta_1^1,\dots,\theta_p^1)$ (where $d \in C_1^p$, without losing generality). Then, $\psi_2 = d{'}(\theta_1^2,\dots,\theta_p^2)$, with $d{'} \in C_1^p \cup C_2^p$, $\theta_i^1 \equiv_{(c_1,c_2)} \theta_i^2$, for every $i = 1,\dots,k$.
Let us prove that $\psi_1(\vec{a}) \in D_{\circledast}$ implies $\psi_2(\vec{a}) \in D_{\circledast}$ (we are using the same notational convention of Proposition \ref{sharing-tecnico-2}). The proof will be done by the following analysis of cases:

\noindent {\bf Case 1}: $d = c_1$. So, $p = k$ and $\psi_1(\vec{a}) = c_1(\ast_{\mu}(\theta_1^1(\vec{a})),\dots, \ast_{\mu}(\theta_k^1(\vec{a})))$ $(\star)$. 

\noindent {\bf Subcase 1.1}: $d{'} = c_2$. Since 
$\psi_1(\vec{a}) \in D_1$ and $(\lambda,\mu)$ is admissible, 
$\lambda(\psi_1(\vec{a})) \in D_2$. In addition, using $(\star)$, $\lambda(\psi_1(\vec{a}))$ = 
$c_2(\lambda(\ast_{\mu}(\theta_1^1(\vec{a}))),\dots,\lambda(\ast_{\mu}(\theta_k^1(\vec{a}))))$. Thus, $c_2(\ast_{\lambda}(\theta_1^1(\vec{a})),\dots,\ast_{\lambda}(\theta_k^1(\vec{a}))) \in D_2$, by Proposition \ref{sharing-tecnico-1}. By the hypotheses and Proposition \ref{sharing-tecnico-2}, we have that 
$c_2(\ast_{\lambda}(\theta_1^2(\vec{a})),\dots,\ast_{\lambda}(\theta_k^2(\vec{a})))$ = 
$c_2(\theta_1^2(\vec{a}),\dots,\theta_k^2(\vec{a})) = \psi_2(\vec{a}) \in D_2 \su D_{\circledast}$.

\noindent {\bf Subcase 1.2}: $d{'} = c_1$.  By Proposition \ref{sharing-tecnico-2} and $(\star)$ again, we have here that $\psi_1(\vec{a})$ = 
$c_1(\ast_{\mu}(\theta_1^2(\vec{a})),\dots, \ast_{\mu}(\theta_k^2(\vec{a}))) = \psi_2(\vec{a})$. So 
$\psi_1(\vec{a}) \in D_{\circledast}$ implies $\psi_2(\vec{a}) \in D_{\circledast}$, obviously.

\noindent {\bf Case 2}: $d \neq c_1$. So, $\psi_2 = d(\theta_1^2,\dots,\theta_p^2)$ and then $\psi_1(\vec{a})$ = 
$d(\theta_1^1(\vec{a}),\dots,\theta_p^1(\vec{a})) = d(\ast_{\mu}(\theta_1^1(\vec{a})),\dots,\ast_{\mu}(\theta_p^1(\vec{a})))$ = 
$d(\ast_{\mu}(\theta_1^2(\vec{a})),\dots,\ast_{\mu}(\theta_p^2(\vec{a})))$ =

\noindent $d(\theta_1^2(\vec{a}),\dots,\theta_p^2(\vec{a})) = \psi_2(\vec{a})$ and our claim is therefore valid.
 The converse 
of the already proved implication 
is similar.
\hfill $\Box$

\begin{example} \label{ejemplo-InPk-sharing} Recall from Example \ref{InPk} the fibring by functions of the logics $I^n$ and $P^k$. Now, let us try to apply identification of connectives in such fibring. From Theorem \ref{sharing-fundamental-1} it suffices to ask that $\lambda_i$ and $\mu_j$ be bijective, $\lambda_i^{-1} = \mu_j$.
 But note that, if $n \geq 1$, then $\lambda_i$ is not injective. A similar fact happens whith $\mu_j$ when $k \geq 1$.
Now, when $n = k = 0$ we have only the
admissible pair $(\lambda_0,\mu_0)$, and since it is easy to verify that $\lambda_0: M/_{\rar_{I^0}} \, \lra M \, /_{\rar_{P^0}}$ is a matrix isomorphism and that ${\lambda_0}^{-1} = \mu_0$, then $\rar_{P_0}$ can be identified with $\rar_{I^0}$. A similar result is valid for $\lambda_0: M / \neg_{I^0} \lra M / \neg_{P^0}$. So, $\neg_{I^0}$ and $\neg_{P^0}$  can be also identified (by the way, the pair $(\lambda_0, \mu_0)$ determines a conservative fibring, according to Proposition~\ref{extendcons}).  The truth-tables of the matrix $M_{(\lambda_0,\mu_0)}$ are particular cases of Example \ref{InPk}. The resulting $C_{\circledast}$-matrix
 $M_{(\lambda_0,\mu_0}$, with $C_{\circledast} = \{\n_{I_0}, \n_{P^0}, \rar_{I_0}, \rar_{P_0}\}$ is just the ``classical logic $CL$ with duplicated truth-values''. Anyway, $C_{\circledast}$ can be reduced, according to previous comments, in such a way that $C_{\circledast} = \{\n_{I_0}, \rar_{I_0}\}$, for instance.
\hfill $\bsq$  
\end{example}

\noindent Let us see now some examples of fibring applied to logics that are better known in the literature. The next one is an analysis of the fibring of the $3$-valued \L ukasiewicz and G\"odel logics. 

The definitions given in the sequel are based on \cite{haj:98}.

\begin{example} \label{luka-godel}
Consider the conjunctive fragments of the logics $\L_3$ and $G_3$ (of \L ukasiewicz and G\"odel, resp.). To simplify notation, $\L_3$ will be denoted by $\lc_a$ and $G_3$ by $\lc_b$ (these subscripts will be applied to all the components of $\L_3$ and $G_3$, such as signatures, matrices, truth-functions and so on). 
The involved signatures are $C_{a}$:=$\{\&_{a} \}$ and $C_b$:=$\{\&_{b}\}$, that are interpreted in the matrices 
$M_{\&_a} = ({\bf A}_{a},D_{a})$ and $M_{\&_b} = ({\bf A}_{b},D_{b})$, with $A_a:= \{ 0_a, {\frac{1}{2}}_a, 1_a \}$, $D_a$:=$\{1_a\}$, $A_b$:= 
$\{ 0_b, {\frac{1}{2}}_b, 1_b \}$, $D_b = \{1_b\}$. The respective truth-funcions are indicated below.

$$
\begin{array}{|c|c c c|} \hline
\&_{a}  & 0_{a} & {\frac{1}{2}}_{a} & 1_{a} \\ \hline 0_{a} & 0_{a} & 0_{a} & 0_{a} \\
{\frac{1}{2}}_{a} & 0_{a} & 0_{a} & {\frac{1}{2}}_{a}\\
1_{a} & 0_{a} & {\frac{1}{2}}_{a} & 1_{a} \\ \hline \end{array}
\hspace{2cm}
\begin{array}{|c|c c c|} \hline
\&_{b}  & 0_{b} & {\frac{1}{2}}_{b} & 1_{b} \\ \hline 0_{b} & 0_{b} & 0_{b} & 0_{b} \\
{\frac{1}{2}}_{b} & 0_{b} & {\frac{1}{2}}_{b} & {\frac{1}{2}}_{b}\\
1_{b} & 0_{b} & {\frac{1}{2}}_{b} & 1_{b} \\ \hline
\end{array}$$

\

\noindent Suppose now that we wish to apply Theorem \ref{sharing-fundamental-1} to obtain fibring pairs $(\lambda,\mu)$ in such a way that $\&_a$ and $\&_b$ could be identified. We would restrict our analysis to the following maps (which will determine $4$ fibring pairs):

\

\bc

$\lambda_1(0_a) = 0_b$; $\lambda_1({\frac{1}{2}}_a) = {\frac{1}{2}}_b$; $\lambda_1(1_a) = 1_b$.

\

$\lambda_2(0_a) = {\frac{1}{2}}_b$; $\lambda_2({\frac{1}{2}}_a) = 0_b$; $\lambda_2(1_a) = 1_b$.

\

$\mu_1(0_b) = 0_a$; $\mu_1({\frac{1}{2}}_b) = {\frac{1}{2}}_a$; $\mu_1(1_b) = 1_a$.

\

$\mu_2(0_b) = {\frac{1}{2}}_a$; $\mu_2(\dps {\frac{1}{2}}_b) = 0_a$; $\mu_2(1_b) = 1_a$.

\ec

\noindent So, the  four admissible fibring pairs are $(\lambda_1, \mu_1)$, $(\lambda_2, \mu_2)$, $(\lambda_1, \mu_2)$ and $(\lambda_2,\mu_1)$. Every one of these pairs defines a different matrix interpreting two connectives (since $C_{\circledast} = \{\&_a, \&_b \}$) and having six elements in its domain. Let us comment some properties of each matrix obtained by each fibring pair: 

\noindent $i)$ $M_{(\lambda_1,\mu_1)}$: this matrix, obviously, satisfies the conditions of Theorem \ref{sharing-fundamental-1}. So, the connectives $\&_a$ and $\&_b$ can be considered the same one. 

\noindent $ii)$ $M_{(\lambda_2,\mu_2)}$: despite having that $\lambda_2$ is bijective and ${\lambda_2}^{-1} = \mu_2$, $\lambda_2$ is not a matrix homomorphism. In fact, $\lambda_2(0_a \, \&_a \, {\frac{1}{2}}_a)$ =  ${\frac{1}{2}}_b$ and $ \lambda_2(0_a)\, \&_b \, \lambda_2({\frac{1}{2}}_a) = 0_b$. 
However, $(\lambda_2, \mu_2)$ can identify $\&_a$ with $\&_b$. 
To see this, consider first the respective truth-tables in that matrix, which are displayed above.
$$
\begin{array}{|c|c c c c c c |} \hline
\&_{a}  & 0_{a} & {\frac{1}{2}}_{a} & 1_{a} & 0_{b} & {\frac{1}{2}}_{b} & 1_{b} \\ \hline 
0_{a} & 0_{a} & 0_{a} & 0_{a} & 0_a & 0_a & 0_a \\
{\frac{1}{2}}_{a} & 0_{a} & {0}_{a} & {\frac{1}{2}}_{a} & 0_a & 0_a & {\frac{1}{2}}_{a}\\
1_{a} & 0_{a} & {\frac{1}{2}}_{a} & 1_{a} &  {\frac{1}{2}}_{a}  & 0_{a} & 1_a \\ 
0_{b} & 0_{a} & 0_{a} &  {\frac{1}{2}}_{a}  & 0_a & 0_a &  {\frac{1}{2}}_{a} \\
{\frac{1}{2}}_{b} & 0_{a} & 0_a & 0_{a} & 0_a & 0_a & 0_{a}\\
1_{b} & 0_{a} & {\frac{1}{2}}_{a} & 1_{a} &  {\frac{1}{2}}_{a}  & 0_{a} & 1_a \\
\hline
\end{array} \hspace{1.2 cm}
\begin{array}{|c|c c c c c c |} \hline
\&_{b}  & 0_{a} & {\frac{1}{2}}_{a} & 1_{a} & 0_{b} & {\frac{1}{2}}_{b} & 1_{b} \\ \hline 
0_{a} & {\frac{1}{2}}_{b} & 0_{b} & {\frac{1}{2}}_{b} & 0_b & {\frac{1}{2}}_{b} & {\frac{1}{2}}_{b} \\
{\frac{1}{2}}_{a} & 0_{b} & 0_{b} & 0_{b} & 0_b & 0_{b} & 0_{b}\\
1_{a} & {\frac{1}{2}}_{b} & 0_{b} & 1_{b} & 0_b & {\frac{1}{2}}_{b} & 1_b \\ 
0_{b} & 0_{b} & 0_{b} & 0_{b} & 0_b & 0_b & 0_b \\
{\frac{1}{2}}_{b} & {\frac{1}{2}}_{b} & 0_{b} & {\frac{1}{2}}_{b} & 0_b & {\frac{1}{2}}_{b} & {\frac{1}{2}}_{b}\\
1_{b} & {\frac{1}{2}}_{b} & 0_{b} & 1_{b} & 0_b & {\frac{1}{2}}_{b} & 1_b \\
\hline
\end{array}$$

\noindent Now, it can be easily proven, by induction on the number of connectives in $\psi_1$, the following fact:

\noindent $(\star)$ If $\psi_1 \equiv_{(\&_a,\&_b)} \psi_2$, then for every $\vec{a} \in A_{\circledast}$, $\psi_1(\vec{a}) \in D_{\circledast}$ iff $\psi_2(\vec{a}) \in D_{\circledast}$ (as before, $\psi_i(\vec{a})$ abbreviates $\psi_i^{\bf A_{\circledast}^{(\lambda,\mu)}}(\vec{a})$).

That is, even when Theorem \ref{sharing-fundamental-1} cannot be applied, $\&_a$ and $\&_b$ can be shared in $[\lc_a \circledast \lc_b]_{(\lambda_2,\mu_2)}$. 

\noindent $iii)$ The matrix $M_{\lambda_1,\mu_2}$: since $\lambda ^{-1} \neq \mu_2$, the conditions of Theorem \ref{sharing-fundamental-1} are not satisfied. But, as in $ii)$, it can be shown that $\&_a$ and $\&_b$ can be identified.

\noindent $iv)$ The matrix $M_{(\lambda_2,\mu_1)}$: the analysis of this matrix is as in $M_{(\lambda_1,\mu_2)}$. 
\hfill $\bsq$
\end{example}

\noindent The previous example shows that the conditions of Theorem \ref{sharing-fundamental-1}, sufficient for the identification of two connectives, are not neccesary. However, this example is based on very simple languages (with just one connective). To realize the importance of this fact, we conclude this section with a variant of Example \ref{luka-godel} that shows that, in the analysis of the identification of two connectives $c_1$ and $c_2$, we must consider  {\em all the connectives of} $C_{\circledast}$, and not merely $c_1$ and $c_2$.

\begin{example}\label{luka-godel-implicacion}
Consider $\lc_a{'}$ and $\lc_b{'}$ the  {\em conjunctive-implicative} fragments of \L ukasiewicz and G\"odel three-valued logics, respectively. That is, $C_a = \{\&_a,\rar_a\}$, $C_b = \{\&_b, \rar_b\}$ (and therefore 
$C_{\circledast}^{'} = \{\&_a, \&_b, \to_a, \rar_b\}$). In addition, $A_a$ and $A_b$ are as above, and the truth-functions $\&_a$ and $\&_b$ are the same of Example \ref{luka-godel}. With respect to the implications, their associated truth-functions are:

$$
\begin{array}{|c|c c c|} \hline
\rightarrow_{a} & 0_{a} & {\frac{1}{2}}_{a} & 1_{a} \\ \hline 0_{a} & 1_{a} & 1_{a} & 1_{a} \\
{\frac{1}{2}}_{a} & {\frac{1}{2}}_{a} & 1_{a} & 1_{a}\\
1_{a} & 0_{a} & {\frac{1}{2}}_{a} & 1_{a} \\ \hline \end{array}
\hm
\begin{array}{|c|c c c|} \hline
\rightarrow_b & 0_{b} & {\frac{1}{2}}_{b} & 1_{b} \\ \hline 0_{b} & 1_{b} & 1_{b} & 1_{b} \\
{\frac{1}{2}}_{b} & {0}_{b} & 1_{b} & 1_{b}\\
1_{b} & 0_{b} & {\frac{1}{2}}_{b} & 1_{b} \\ \hline \end{array}$$

\

\noindent Consider now the fibring 
$[\lc_1{'} \circledast \lc_2{'}]_{(\lambda_2,\mu_2)}$, determined by the $C_{\circledast}^{'}$-matrix $M_{(\lambda_2,\mu_2)}^{'} = ({\bf A}_{\circledast}^{(\lambda_2,\mu_2)}, D_{\circledast})$: we already have seen that $\lambda_2: M /_{\&_a} \lra M/_{\&_b} $ is not a matrix isomorphism\footnote{Moreover, it can be shown that $\lambda_2$ is not a matrix homomorphism from $M /_{\rar_a}$ to $M /_{\rar_b}$.}, but $\&_a$ and $\&_b$ can be identified within the context of $\lc_a$ and $\lc_b$. But that is not the case if we consider $\lc_a{'}$ and $\lc_b{'}$: let us consider $\psi_1:= p \rar_a (q \, \&_a \, r)$ and $\psi_2$ := $p \rar_a (q \, \&_b \, r)$. Obviously, $\psi_1 \equiv_{(\&_a,\&_b)} \psi_2$. Now, suppose $\vec{a} = ({\frac{1}{2}}_a,0_b,0_a)$, and consider the following tables describing the truth-functions $\rightarrow_a$ and $\rightarrow_b$ in $M_{(\lambda_2,\mu_2)}^{'}$:
$$
\begin{array}{|c|c c c c c c|} \hline
\rightarrow_a & 0_{a} & {\frac{1}{2}}_{a} & 1_{a} & 0_{b} & {\frac{1}{2}}_{b} & 1_{b}\\ \hline 
0_{a} & 1_{a} & 1_{a} & 1_{a} & 1_a & 1_a & 1_a \\
{\frac{1}{2}}_{a} & {\frac{1}{2}}_{a} & 1_{a} & 1_{a} & 1_{a} & {\frac{1}{2}}_{a} & 1_a \\
1_{a} & 0_{a} & {\frac{1}{2}}_{a} & 1_{a} & {\frac{1}{2}}_{a} & 0_{a} & 1_a \\
0_{b} & {\frac{1}{2}}_{a} & 1_{a} & 1_{a} &1_{a} & {\frac{1}{2}}_{a} & 1_a \\
{\frac{1}{2}}_{b} & {1}_{a} & 1_{a} & 1_{a} & 1_a & 1_a & 1_a \\
1_{b} & 0_{a} & {\frac{1}{2}}_{a} & 1_{a} & {\frac{1}{2}}_{a} & 0_a & 1_a \\ \hline \end{array} \hspace{1.2 cm}
\begin{array}{|c|c c c c c c|} \hline
\rightarrow_b & 0_{a} & {\frac{1}{2}}_{a} & 1_{a} & 0_{b} & {\frac{1}{2}}_{b} & 1_{b}\\ \hline 
0_{a} & 1_{b} & 0_{b} & 1_{b} & 0_b & 1_b & 1_b \\
{\frac{1}{2}}_{a} & 1_b & 1_{b} & 1_{b} & 1_b & 1_b & 1_b \\
1_{a} & {\frac{1}{2}}_{b} & 0_{b} & 1_{b} & 0_b & {\frac{1}{2}}_{b} & 1_b \\ 
0_{b} & 1_{b} & 1_{b} & 1_{b} & 1_b & 1_b & 1_b \\
{\frac{1}{2}}_{b} & 1_b & 0_{b} & 1_{b} & 0_b & 1_b & 1_b \\
1_{b} & {\frac{1}{2}}_{b} & 0_{b} & 1_{b} & 0_b & {\frac{1}{2}}_{b} & 1_b \\ \hline \end{array}
$$
\noindent From all the truth-tables displayed above, $\psi_1(\vec{a})$ = 
${\frac{1}{2}}_a \rar_a (0_b \, \&_a \, 0_a) = {\frac{1}{2}}_a \rar_a 0_a = {\frac{1}{2}}_a \notin D_{\circledast}$, meanwhile $\psi_2(\vec{a})$ = 
${\frac{1}{2}}_a \rar_a (0_b \, \&_b \, 0_a) = {\frac{1}{2}}_a \rar_a 0_b = 1_a \in D_{\circledast}$. So, $\&_a$ cannot be identified with $\&_b$ in this case. 
\hfill $\bsq$
\end{example}

\section{Concluding Remarks}\label{final}

Summarizing the results presented in this paper, we have introduced certain techniques for the combination of
logics, by adapting the original formulation of
fibring for modal logics, in this case applied to matrix logics. In this sense, fibring by functions is a simple
generalization, which produces a
weak conservative extension of the given logics, and it can easily be explained by means of the functions $\ast_{\lambda}$ and $\ast_{\mu}$. In order to
obtain a strong extension of the given logics (as it is usually expected for a ``good'' combination technique) it is necessary to
impose additional conditions to the fibring pairs. This originates the so-called strong fibring of logics $[\lc_1 \odot \lc_2]_{(\lambda, \mu)}$ which is, moreover, a conservative extension of both logics. 

On the other hand, we gave special attention to the problem of the identification of connectives, according to similar approaches in the literature (see \cite{ser:ser:cal:99} or \cite[Chapters 5-8]{car:con:gab:gou:ser:08}, where the techniques for constrained fibring are extensively studied, but under a categorical approach). It is  precisely the study of the identification of connectives which provide us some very interesting problems to be analyzed in the future. Among others:
\begin{itemize}
  \item Theorem \ref{sharing-fundamental-1} shows a sufficient condition that guarantees that a fibring pair $(\lambda, \mu)$ can identify two given connectives and, in Example~\ref{luka-godel}, it was shown that this condition is not necessary. An obvious problem here is to find   necessary  and sufficient conditions that characterize sharing of connectives.
  \item As we have remarked, once two connectives are shared, we could eliminate one of them. Which is the mathematical treatment that allows us to explain that elimination, within the context of matrix logics? And, supposing that a reasonable formalism for this process is found, which interesting results are valid from it?
\end{itemize}

Note here that the process of identification, explained in this paper for the treatment of just one 
connective, could be generalized for the identification of   sets of connectives. We conjecture that this generalization is natural enough, and it will be treated in further studies.

Finally, we believe that the techniques here introduced can be
applied to obtain new interesting logics by combination, as it was done in Example~\ref{InPk}. This example, which studies the fibring of the logics $I^n$ and $P^k$ is specially interesting for us, because in previous works we have defined and analyzed the  {\em $I^nP^k$-logics}, by means of other definitions that are not based on fibring (see~\cite{fer:01}, \cite{fer:18} and~\cite{fer:22}). So, a natural question to be studied here is if the mentioned logics $I^n P^k$ can actually be defined by means of $[I^n \circledast P^k]_{(\lambda,\mu)}$,  the fibring  of a specific fibring pair 
$(\lambda,\mu)$. Another interesting class of matrix logics to be studied are in the realm of mathematical fuzzy logic. To this respect, some basic examples were analyzed here in Example~\ref{luka-godel} and, in~\cite{con:fer:05}, we have already presented some results about combinations of fuzzy logics. We believe that mathematical fuzzy logic could be an interesting source of examples.

Finally, the proof-theoretic counterpart of fibring by functions, by means  of Hilbert calculi or sequent calculi, is without doubts a challenging topic that deserves to be studied. In particular, the associated problem of completeness preservation by fibring by functions is one of the topics that should be addressed.

\section*{ Acknowledgements:}  
The first author acknowledges the supports from CICITCA - UNSJ (code F1191) and from CONICET, Argentina (PIP 11220200100912-CO). \\
The second author acknowledges financial
support from the National Council for Scientific and Technological Development (CNPq), Brazil,
under research grant 306530/2019-8.

\bibliographystyle{plain}

\end{document}